\documentclass[12pt]{amsart}

\usepackage{tgpagella}
\usepackage{euler}
\usepackage[T1]{fontenc}
\usepackage{amsmath, amssymb, mathabx}
\usepackage[hidelinks]{hyperref}
\usepackage[english]{babel}
\usepackage{mathrsfs}
\usepackage{eucal}
\usepackage[all]{xy}
\usepackage{tikz}

\newtheorem{thm}{Theorem}[section]
\newtheorem*{thm*}{Theorem}
\newtheorem{lem}[thm]{Lemma}
\newtheorem{fact}[thm]{Fact}

\newtheorem{prop}[thm]{Proposition}
\newtheorem*{prop*}{Proposition}

\newtheorem{cor}[thm]{Corollary}
\newtheorem*{cor*}{Corollary}

\theoremstyle{definition}
\newtheorem{defn}[thm]{Definition}
\newtheorem*{defn*}{Definition}

\newtheorem{remark}[thm]{Remark}

\newtheorem{question}[thm]{Question}

\newtheorem*{question*}{Question}
\newtheorem*{Pquestion*}{Popa's question}

\newtheorem*{conv*}{Convention}

\def\bb{\mathbb}

\def\bb{\mathbb}

\def\cal{\mathcal}

\def\u{\mathsf 1}

\newcommand{\cstar}{$\mathrm{C}^*$}

\makeatletter

\def\dotminussym#1#2{%
  \setbox0=\hbox{$\m@th#1-$}%
  \kern.5\wd0%
  \hbox to 0pt{\hss\hbox{$\m@th#1-$}\hss}%
  \raise.6\ht0\hbox to 0pt{\hss$\m@th#1.$\hss}%
  \kern.5\wd0}
\newcommand{\dotminus}{\mathbin{\mathpalette\dotminussym{}}}

\def \Th{\operatorname{Th}}
\def \R{\mathcal R}
\def \u{\mathcal U}

\def \val{\operatorname{val}}

\def \mre{\operatorname{MIP}^*=\operatorname{RE}}

\def \sub{\operatorname{Sub}}
\def \irs{\operatorname{IRS}}
\def \sof{\operatorname{sof}}
\def \erg{\operatorname{erg}}

\allowdisplaybreaks
\makeatletter
\def\l@subsection{\@tocline{2}{0pt}{2.5pc}{5pc}{}}
\def\l@subsubsection{\@tocline{2}{0pt}{5pc}{7.5pc}{}}
\makeatother

\begin{document}


\title[Incompleteness in operator algebras]{Undecidability and incompleteness in quantum information theory and operator algebras}

\author{Isaac Goldbring}
\address{Department of Mathematics\\University of California, Irvine, 340 Rowland Hall (Bldg.\# 400),
Irvine, CA 92697-3875}
\email{isaac@math.uci.edu}
\urladdr{http://www.math.uci.edu/~isaac}
\thanks{Goldbring was partially supported by NSF grant DMS-2054477.}

\begin{abstract}
We survey a number of incompleteness results in operator algebras stemming from the recent undecidability result in quantum complexity theory known as $\mre$, the most prominent of which is the G\"odelian refutation of the Connes Embedding Problem.  We also discuss the very recent use of $\mre$ in refuting the Aldous-Lyons conjecture in probability theory.
\end{abstract}
\dedicatory{Dedicated to Kurt G\"odel on the 100th anniversary of his matriculation at the University of Vienna.}

\maketitle



\section{Introduction}
Two of G\"odel's most famous results are his \emph{Completeness theorem} for first-order logic and his \emph{Incompleteness theorem} for first-order arithmetic.  In this article, we survey how a version of the Completeness theorem for continuous logic has recently been used, together with a major undecidability result in quantum information theory, to prove a number of incompleteness results in operator algebras.  (Some basic definitions and facts about operator algebras can be found in Section 3.)

The main incompleteness result is connected to aguably one of the most famous open problems in operator algebras, namely the \textbf{Connes embedding problem} (CEP), which first arose in Connes' landmark 1976 paper \cite{Connes}, where he wrote the following:  

\emph{We now construct an approximate imbedding of $N$ in $\R$. Apparently such an imbedding ought to exist for all II$_1$ factors because it does for the regular representation of free groups. However, the construction below relies on condition 6.}

Here, $\R$ is a particularly important von Neumann algebra known as the \textbf{hyperfinite II$_1$ factor}.  In the paper, Connes shows how to ``approximately embed'' some particular tracial von Neumann algebra $N$ into $\R$, which he later shows means that he can embed $N$ into an ultrapower of $\R$.  He then comments that, although his proof relies on a particular condition satisfied by the algebra $N$, this embedding out to exist regardless of this extra assumption, the reason being that the embedding exists for another particularly important algebra, namely the group von Neumann algebra corresponding to the free group.  The relative uncertainty as to the strength of his conviction (given the use of the words ``ought to'') combined with the fairly weak justification regarding the free group factor, has led many to call this statement a \emph{problem} rather than a \emph{conjecture}.

The CEP was later shown to have (surpring) connections to a number of disparate areas of mathematics, including \cstar-algebra theory, free probability, logic, and quantum information theory to name a few.  It was the latter connection that ultimately led to its refutation in 2020, when a result in quantum complexity theory known as $\mre$ \cite{MIP} was announced.  Roughly speaking, $\mre$ implies that the sets of probability distributions arising from measuring entangled quantum states are incredibly complicated (from the point of view of computability theory); a detailed discussion of the result, together with the necessary background material on quantum correlations, can be found in Section 2 of the paper.  The connection between logic and the CEP is rather straightforward:  in the language appropriate for studying tracial von Neumann algebras, the CEP states that every tracial von Neumann algebra is a model of the universal theory of $\R$.

The path from $\mre$ to the refutation of the CEP is highly nontrivial and involves several deep results.  That being said, it was shown by the author and Hart in \cite{compR} and \cite{R} that by using a version of G\"odel's Completeness theorem for continuous logic together with certain operator algebraic interpretations of the aforementioned correlation sets, one can give a refutation of CEP from $\mre$ that avoids these intermediate results.  Moreover, just as G\"odel's incompleteness theorem for arithmetic shows that no satisfiable c.e. subset of true arithmetic can axiomatize true arithmetic, the model-theoretic approach mentioened above yields an Incompleteness theorem for the universal theory of $\R$:  there is no satisfiable c.e. subtheory of the universal (or even full) theory of $\R$ that axiomatizes the universal theory of $\R$.  In \cite{R}, we deemed this result a G\"odelian refutation of the CEP.  Besides being interesting from the logical perspective, this G\"odelian refutation of the CEP had a number of interesting purely operator algebraic consequences, as was outlined in \cite{R}.  This G\"odelian refutation of the CEP is described in Section 4 of the paper.   

The G\"odelian refutation of the CEP led also to a number of other undecidability/incompleteness results in operator algebras \cite{QWEPundec, tsirelson} which we survey in Sections 4 and 5 of the paper.

Only a couple of months before the writing of this article, the ideas around $\mre$ were used to refute a prominent conjecture in probability theory known as the \textbf{Aldous-Lyons conjecture}.  Although this resolution does not involve model theory per se\footnote{At least for the moment...}, we could not resist including a brief discussion of this conjecture and its resolution in Section 6.

A number of open questions related to the themes of the paper are collected in Section 7.

It is well-known that G\"odel became great friends with Einstein during the latter part of their careers while at the Institute for Advanced Study; G\"odel was very interested in physics and even dabbled in general relativity, constructing a pathological solution to Einstein's equations that contained closed timelike loops.  (See \cite{godeleinstein} for an interesting popular account of both this friendship and solution to Einstein's equations.)  We can only speculate that, around 70 years after his original Completeness and Incompleteness Theorems were published, G\"odel would have been thrilled to see his ideas being pursued in these areas of mathematics and computation so closely tied to physics.   

Throughout this article, we assume familiarity with basic continuous logic.  The reader unfamiliar with continuous logic is encouraged to refer to Hart's article \cite{Hart}, which is written with an eye towards applications to operator algebras.

The reader interested in learning more about the CEP and its connections with logic, quantum information theory, and \cstar-algebras can consult the author's (fairly comprehensive) survey \cite{bulletin}.

One bit of notation that we employ throughout the paper:  given a positive integer $n$, we set $[n]:=\{1,\ldots,n\}$.

We thank Matthias Aschenbrenner, Michael Chapman and Henry Yuen for very helpful conversations during the writing of this paper.

\section{The main undecidability result}

In this section, we state the main undecidability result on which all the others rest, namely the result known as $\mre$.  First, we explain the requisite terminology from quantum mechanics.

\subsection{Quantum correlation sets}

The main undecidability result concerns the complexity of the sets of conditional probabilities arising from measuring entangled quantum states.  Before we can define these sets, we need to review how quantum measurements work.

Suppose that Alice is a scientist studying some (quantum) physical system.  She represents various states of this system by unit vectors in some (for now, finite-dimensional) Hilbert space $H_A$.  The mathematical object used in connection with measuring some particular observable of the system is a sequence of operators $\vec A=(A_1,\ldots,A_n)$ on $H_A$ associated with the possible outcomes of the measurement; according to the Born rule of quantum mechanics, the probability that the $i^{\text{th}}$ possible outcome occurs when the system is measured in the state given by the unit vector $\xi$ is given by $\langle A_i\xi,\xi\rangle$.  In order for these quantities to represent probabilities, we need $0\leq \langle A_i\xi,\xi\rangle\leq 1$ for all unit vectors $\xi\in H_A$ and all $i\in [n]$.  Moreover, given any unit vector $\xi$, the probability of some outcome occurring must be $1$, corresponding to the requirement that $\sum_{i=1}^n\langle A_i\xi,\xi\rangle=1$.  Given that these two requirements must hold for all unit vectors $\xi$, we see that the sequence $\vec A=(A_1,\ldots,A_n)$ of operators on $H_A$ must satisfy the following two requirements:
\begin{enumerate}
    \item Each $A_i$ is a \emph{positive operator}, meaning that $\langle A_i\xi,\xi\rangle\geq 0$ for all $\xi\in H_A$.
    \item $\sum_{i=1}^n A_i=I_A$, where $I_A$ is the identity operator on $H_A$.
\end{enumerate}
We call a sequence $\vec A=(A_1,\ldots,A_n)$ satisfying the above two requirements a \textbf{positive operator valued measure (POVM) of length $n$ on $H_A$}.

Things get more interesting when a second scientist Bob enters the picture.  States of Bob's physical system are represented by unit vectors in some (again, for now, finite-dimensional) Hilbert space $H_B$.  The \textbf{composite system} consisting of Alice's system and Bob's system is modeled by the tensor product Hilbert space $H_A\otimes H_B$.\footnote{For those readers unfamiliar with the tensor product of two Hilbert spaces, perhaps the simplest description of this space is that an orthonormal basis for $H_A\otimes H_B$ is given by vectors of the form $e\otimes f$, where $e$ ranges over an orthonormal basis for $H_A$ and $f$ ranges over an orthonormal basis for $H_B$.}  If the state of Alice's system is modeled by the unit vector $\xi_A\in H_A$ and the state of Bob's system is modeled by the unit vector $\xi_B\in H_B$, then the state of the composite system is modeled by the unit vector $\xi_A\otimes \xi_B\in H_A\otimes H_B$.  However, there are unit vectors $\xi\in H_A\otimes H_B$ that are not of this form; in this case, the corresponding state is said to be \textbf{entangled}.  When a composite system is in an entangled state, neither of the constitutent systems can be thought of as having a definite state of their own.  

Given operators $A$ and $B$ on $H_A$ and $H_B$ respectively, one can define a new operator $A\otimes B$ on $H_A\otimes H_B$ whose action on elementary tensors is given by $(A\otimes B)(\xi_A\otimes \xi_B):=A\xi_A\otimes B\xi_B$.  If each of $A$ and $B$ are positive, then so is $A\otimes B$.  If $\vec A:=(A_1,\ldots,A_n)$ and $\vec B:=(B_1,\ldots,B_n)$ are POVMs on $H_A$ and $H_B$ respectively of length $n$, then one obtains a POVM $\vec A\otimes \vec B$ of length $n^2$ on $H_A\otimes H_B$ consisting of the operators $A_i\otimes B_j$ for $i,j\in [n]$.

We are now ready to define the main sets of interest in this paper.  Fix integers $k,n\geq 2$.  Suppose that $H_A$ and $H_B$ are finite-dimensional Hilbert spaces and that, for each $x,y\in [k]$, there are POVMs $\vec A^x:=(A^x_1,\ldots,A^x_n)$ and $\vec B^y:=(B^y_1,\ldots,B^y_n)$ of length $n$ on $H_A$ and $H_B$ respectively.  Further suppose that $\xi\in H_A\otimes H_B$ is a unit vector.  Consider the probability $$p_{\vec A,\vec B,\xi}(a,b|x,y):=\langle (A^x_a\otimes B^y_b)\xi,\xi\rangle$$ of obtaining outcomes $a$ and $b$ respectively when measuring with the $x^{\text{th}}$ POVM on $H_A$ and the $y^{\text{th}}$ POVM on $H_B$, assuming the composite system is in the state corresponding to the unit vector $\xi$.  This collection of numbers forms a $kn\times kn$-matrix, which we can also view as a vector $p_{\vec A,\vec B,\xi}\in [0,1]^{k^2n^2}$.  We call such a vector $p$ a \textbf{quantum correlation}.  The collection of quantum correlations, where there are $k$  measurement operators on each space, each of which have $n$ possible outcomes, is denoted $C_q(k,n)$.  It is important to note that while $H_A$ and $H_B$ are required to be finite-dimensional, there is no upper bound imposed on these dimensions.

It is readily verified that each set $C_q(k,n)$ is a convex subset of $[0,1]^{k^2n^2}$.  An old problem of Tsirelson (sometimes called the \emph{weak Tsirelson problem}, as opposed to the (full) Tsirelson problem, which we will encounter later in this section) asks whether or not $C_q(k,n)$ is a closed subset of $[0,1]^{k^2n^2}$.  Only relatively recently a negative resolution to this problem was given by Slofstra in \cite{slofstra}.  We consider the closure $C_{qa}(k,n):=\overline{C_q(k,n)}$ of $C_q(k,n)$ inside of $[0,1]^{k^2n^2}$; elements of this set are called \textbf{quantum asymptotic correlations}.

The main undecidability result of \cite{MIP} asserts that, in general, as $k$ and $n$ vary, the closed, convex subsets $C_{qa}(k,n)$ are incredibly complicated.  For example, the authors point out in the introduction to \cite{MIP} that, combined with standard results in combinatorial optimization theory (in particular, using the main result of \cite{NY}), there is no effective solution to the \textbf{$\epsilon$-membership problem} for these sets:  given rational $\epsilon \in (0,1)$, there is no algorithm such that, given $k,n\geq 2$ and $p\in [0,1]^{k^2n^2}$ with rational coordinates, determines whether or not $p$ belongs to $C_{qa}(k,n)$ or is at least $\epsilon$ away from all elements of $C_{qa}(k,n)$, promised that one of these is the case.

\subsection{Nonlocal games and $\mre$}

One can use elements of $C_{qa}(k,n)$ as strategies in very simple two-player games called \textbf{nonlocal games}.  Once again, fix integers $k,n\geq 2$.  A nonlocal game with $k$ questions and $n$ answers consists of two parties, say Alice and Bob again, interacting with a Referee.  The game consists of three steps:
\begin{enumerate}
    \item First, the Referee uses a probability distribution on $[k]^2$ to randomly ask Alice and Bob a pair of questions $(x,y)\in [k]^2$.
    \item Second, \emph{somehow} Alice and Bob return a pair of answers $(a,b)\in [n]^2$.  
    \item Finally, to decide whether or not Alice and Bob ``win'' this play of the game, the Referee consults a ``rulebook'' to see if their joint answers to the questions are ``correct.''  More formally, the Referee computes $D(x,y,a,b)$, where $D:[k]^2\times [n]^2\to \{0,1\}$ is some function, called the \textbf{decision predicate} for the game; $D(x,y,a,b)=1$ indicates that they have won this round of the game, while $D(x,y,a,b)=0$ indicates that they have lost.
\end{enumerate}

More formally, a nonlocal game with $k$ questions and $n$ answers is a pair $\frak G=(\pi_{\frak{G}},D_{\frak{G}})$, where $\pi_{\frak{G}}$ is a probability distribution on $[k]^2$ and $D_{\frak{G}}:[k]^2\times [n]^2\to \{0,1\}$ is a function.

The question becomes:  how should Alice and Bob decide to play this game?  To lay down the ground rules, Alice and Bob can discuss ahead of time how they are going to play the game; however, once the game has started, they are not allowed to communicate to each other (in particular, they cannot let the other person know the question that they received).  Moreover, Alice and Bob are ``on the same team''; they want to jointly win each round of the game.  Such players are said to be cooperating but non-communicating.

In its most general form, a strategy for the game would consist of probability distributions $p(a,b|x,y)$ indicating the probability that, when receiving the pair of questions $(x,y)\in [k]^2$, they respond with the pair of answers $(a,b)\in [n]^2$.  Such a strategy is once again naturally an element of $[0,1]^{k^2n^2}$.  Given a strategy $p$, the expected value that Alice and Bob win the game when employing strategy $p$ is given by the quantity 
$$\val(\frak G,p):=\sum_{(x,y)\in [k]^2}\pi_{\frak G}(x,y)\sum_{(a,b)\in [n]^2}p(a,b|x,y)D_{\frak G}(x,y,a,b).$$

One simple-minded strategy for Alice and Bob to follow would be to use a deterministic strategy, meaning that Alice and Bob use functions $A:[k]\to [n]$ and $B:[k]\to [n]$ which tells them how they should respond to each question.  A slightly more sophisticated strategy could incorporate some ``classical randomness'' by, for example, each player deciding on 6 possible deterministic strategies and then rolling a die before the game starts to decide which strategy they will use in that round of the game.  (More generally, they could use any probability space to determine which of their classical strategies to employ.)  The collection of such strategies will be called \textbf{classical correlations} and will be denoted by $C_c(k,n)$.  The maximum expected value of them winning the game when employing a classical strategy is called the \textbf{classical value of $\frak G$} and is given by
$$\val(\frak G):=\sup_{p\in C_c(k,n)}\val(\frak G,p).$$

On the other hand, one may instead consider employing quantum strategies of the form $p=p_{\vec A,\vec B,\xi}$ as above, leading to the \textbf{quantum value of $\frak G$}, given by
$$\val^*(\frak G):=\sup_{p\in C_q(k,n)}\val(\frak G,p)=\sup_{p\in C_{qa}(k,n)}\val(\frak G,p).$$
There is a standard way of viewing a classical correlation as a quantum correlation, whence $C_c(k,n)\subseteq C_q(k,n)$ for all $k,n\geq 2$ and $\val(\frak G)\leq \val^*(\frak G)$ for every nonlocal game $\frak G$.  Bell's famous theorem in quantum mechanics \cite{bell} can be viewed as asserting that, in general, $C_c(k,n)$ is a proper subset of $C_q(k,n)$, which can in turn be used to show that, for certain nonlocal games $\frak G$, $\val(\frak G)<\val^*(\frak G)$ holds.

It is straightforward to verify that the supremum in the definition of $\val(\frak G)$ is indeed a maximum and occurs when employing some deterministic strategy.  Consequently, by employing a brute force search, one sees that the function $\frak G\mapsto \val(\frak G)$ (restricted to games where $\pi_{\frak G}$ takes only rational values) is a computable function.

But what about the function $\frak G\mapsto \val^*(\frak G)$?  Since, a priori, $\val^*(\frak G)$ is irrational, even when $\pi_{\frak G}$ takes only rational values, one must refine the question by asking if there is a computable function which, upon input $\frak G$ and rational $\epsilon>0$, returns an interval $I\subseteq [0,1]$ with rational endpoints containing $\val^*(\frak G)$ and such that $|I|<\epsilon$?

By enumerating denser and denser sets of unit vectors from larger and larger finite-dimensional Hilbert spaces while also enumerating denser and denser sets of POVMs on those spaces, a brute force search allows one to find a computable sequence of lower bounds for $\val^*(\frak G)$, uniformly in $\frak G$.  In computability-theoretic terminology, this says that $\val^*(\frak G)$ is a \textbf{left-c.e.} real number, uniformly in $\frak G$.  The definition of a \textbf{right-c.e.} real number is defined analogously, using a computable sequence of upper bounds.  A real number that is both left-c.e. and right-c.e. is called a \textbf{computable} real number.

The issue with the previous algorithm is that there is no a priori bound on the dimension of the Hilbert spaces involved in the quantum correlations being used as strategies for the games.  As is usual in the subject, perhaps if one waits a sufficiently large time, one will arrive at a dimension where there are strategies which outperform the ones considered thus far.  This is of course analogous to the unsolvability of the halting problem, where perhaps one just needs to be more patient and the program will halt if one simply waits another trillion years.

The main result of \cite{MIP} states that the function $\frak G\mapsto \val^*(\frak G)$ is \emph{not} computable and that the above analogy with the halting problem is appropriate:

\begin{thm}[$\mre$]\label{mip}
There is a computable mapping $\cal M\mapsto \frak G_{\cal M}$ from Turing machines to nonlocal games such that:
\begin{itemize}
    \item If $\cal M$ halts on the empty input, then $\val^*(\frak G_{\cal M})=1$.
    \item If $\cal M$ does not halt on the empty input, then $\val^*(\frak G_{\cal M})\leq \frac{1}{2}$.
\end{itemize}
\end{thm}

A word is in order regarding the name of the previous result.  The reader familiar with computational complexity theory might guess (correctly!) that the names on either side of the equals sign are names of complexity classes.  The class RE consists of those languages (in the sense of complexity theory) that are recursively enumerable.  The class $\operatorname{MIP}^*$ consists of those languages $L$ for which there is an effective mapping $z\mapsto \frak G_z$ from bits to nonlocal games satisfying the conclusion of Theorem \ref{mip} with $\cal M$ halting replaced by $z\in L$ and $\cal M$ not halting replaced by $z\notin L$.  The theorem thus states that the halting problem belongs to $\operatorname{MIP}^*$, when so do all languages in RE by the RE-completeness of the halting problem.  Conversely, using the brute force search method discussed above, any language in $\operatorname{MIP}^*$ belongs to RE, whence the two classes coincide.  MIP stands for ``Multiprover interactive proofs'' while the * indicates that the provers share quantum entanglement.  Besides the fact that it refutes the CEP, $\mre$ is intellectually fascinating for its demonstration of the power of quantum resources in complexity theory:  without the shared quantum resources, languages in $\operatorname{MIP}$ are all computable (and in fact coincide with the class of languages known as NEXP). 

The proof of Theorem \ref{mip} is quite involved and uses a wide variety of techniques.  An extremely readable account of the main idea is given in \cite[Section 2]{MIP}.

\subsection{Tsirelson's problem}

A consequence of Theorem \ref{mip} is a negative solution to a problem in quantum information theory known as \textbf{Tsirelson's problem} \cite{tsirelsonproblem}.  Tsirelson's problem involves an a priori larger class of quantum correlations known as \textbf{quantum commuting correlations}.  Once again, the correlations will arise from Alice and Bob measuring certain observable quantities of a quantum system.  However, this time, it may not be possible to recognize the system as the composite of Alice's part of the system and Bob's part of the system, that is, the Hilbert space $H$ corresponding to states of the system may not be obviously decomposable as a tensor product of the form $H_A\otimes H_B$.  On the other hand, this time we do allow for the Hilbert space $H$ to be infinite-dimensional.  (In fact, if we require $H$ to be finite-dimensional in what follows, then it is known that the set of correlations we are about to define will coincide with the set $C_{qa}(k,n)$ below \cite{needinfdim}.)  However, in order for Alice and Bob to be able to perform their measurements simultaneously, we will require that their corresponding measurement operators commute.

More precisely, fix integers $k,n\geq 2$ as before.  Further suppose that we are given a (possibly infinite-dimensional) Hilbert space $H$, a unit vector $\xi\in H$, and, for each $x,y\in [k]$, POVMs\footnote{Now that $H$ is possibly infinite-dimensional, we also require that each $A^x_a$ and $B^y_b$ is a bounded operator on $H$.} $\vec A^x$ and $\vec B^y$ on $H$ of length $n$ for which $A^x_aB^y_b=B^y_bA^x_a$ for all $x,y\in [k]$ and all $a,b\in [n]$.  The corresponding correlation is given by $p_{\vec A,\vec B,\xi}(a,b|x,y):=\langle A^x_aB^y_b\xi,\xi\rangle$.  We let $C_{qc}(k,n)$ denote the set of all such correlations.  Note that every correlation in $C_q(k,n)$ belongs to $C_{qc}(k,n)$ once we identify the operator $A^x_a$ on $H_A$ with the operator $A^x_a\otimes I_{H_B}$ on $H_A\otimes H_B$ and likewise for the $B^y_b$'s.\footnote{And of course we then identify $p_{\vec A,\vec B,\xi}$ with $p_{\vec A\otimes I_{H_B},I_{H_A}\otimes \vec B,\xi}$.}  If we define the \textbf{quantum commuting value} of a nonlocal game $\frak G$, denoted $\val^{co}(\frak G)$, just as we defined the quantum value of the game except that we now take the supremum over the a priori larger set of quantum commuting correlations, it follows that $\val^*(\frak G)\leq \val^{co}(\frak G)$ for all nonlocal games $\frak G$.

One can show that $C_{qc}(k,n)$ is a closed, convex subset of $[0,1]^{k^2n^2}$, whence it in fact contains $C_{qa}(k,n)$.  Tsirelson's problem asks:  is $C_{qa}(k,n)=C_{qc}(k,n)$ for all $k,n\geq 2$?  Theorem \ref{mip} answers the question negatively:

\begin{cor}\label{tsirelson}
There exists $k,n\geq 2$ for which $C_{qa}(k,n)\not=C_{qc}(k,n)$.  In fact, there is a nonlocal game $\frak G$ for which $\val^*(\frak G)<\val^{co}(\frak G)$.
\end{cor}

\begin{remark}\label{separation}
In \cite[Section 12.3]{MIP}, the authors show how to construct a somewhat explicit example of a correlation that belongs to $C_{qc}(k,n)\setminus C_{qa}(k,n)$.
\end{remark}

Recall that $\val^*(\frak G)$ is left c.e., uniformly in $\frak G$.  On the other hand, we have:  

\begin{prop}\label{qcrightce}
The quantum commuting value $\val^{co}(\frak G)$ is \emph{right}-c.e. uniformly in $\frak G$.
\end{prop}

Corollary \ref{tsirelson} follows immediately from Theorem \ref{mip} and the previous proposition.  Indeed, if $\val^*(\frak G)=\val^{co}(\frak G)$ for all nonlocal games $\frak G$, then this common value would in fact be computable uniformly in $\frak G$, contradicting Theorem \ref{mip}.

The standard proof of Proposition \ref{qcrightce} uses techniques from \textbf{semidefinite programming}; a brief explanation is given in \cite[Section 6.1]{bulletin}.  In Subsection \ref{digression} below, we show how a slight variant of Proposition \ref{qcrightce}, which still suffices to prove Corollary \ref{tsirelson}, follows immediately from the Completeness theorem for continuous logic and a basic operator-algebraic fact.

Through two highly nontrivial intermediate results, the negative solution to Tsirelson's problem also provides a negative solution to the Connes embedding problem in von Neumann algebra theory.  In Section 4, we introduce the CEP and show how the Completeness theorem for continuous logic can be used to bypass these intermediate results while also providing a G\"odelian version of the failure of CEP. 

\section{Primer on operator algebras}

In this section, we provide a very quick introduction to von Neumann algebras and \cstar-algebras.  A more thorough introduction to these topics, aimed at logicians, can be found in the introductory articles \cite{ioana} and \cite{szabo} in the volume \emph{Model theory of operator algebras}. 

Throughout, $H$ denotes a complex Hilbert space and $B(H)$ denotes the set of \textbf{bounded linear operators} $T:H\to H$, where $T$ being bounded means the \textbf{operator norm} $\|T\|:=\sup\{\|T(x)\| \ : \ x\in H, \|x\|\leq 1\}$ is finite.  $B(H)$ has the structure of a $*$-algebra, where $*$ denotes the usual adjoint of linear operators.

A \textbf{\cstar-algebra} is a $*$-subalgebra of $B(H)$ closed in the above operator norm topology.  A \textbf{von Neumann algebra} is a unital $*$-subalgebra of $B(H)$ closed in a topology weaker than the norm topology called the \textbf{strong operator topology} (SOT), where a net $(T_i)_{i\in I}$ SOT converges to an operator $T$ if the net $(T_i(x))_{i\in I}$ converges to $T(x)$ for all $x\in H$.  The above descriptions of \cstar-algebras and von Neumann algebras are the ``concrete'' ones while purely abstract definitions are possible.  The unital commutative \cstar-algebras are (up to isomorphism) exactly the algebras $C(X)$ of continuous functions on a compact space $X$, whence \cstar-algebra theory is sometimes referred to as ``noncommutative topology.''  Similarly, the commutative von Neumann algebras are (up to isomorphism) exactly the algebras $L^\infty(X,\mu)$ of essentially bounded measurable functions on a measure space $(X,\mu)$, whence von Neumann algebra theory is often referred to as ``noncommutative measure theory.''

A \textbf{state} on a unital \cstar-algebra $A$ is a linear functional $\varphi:A\to \mathbb{C}$ satisfying $\varphi(1)=1$ that is also \textbf{positive}, meaning that $\varphi(x^*x)\geq 0$ for all $x\in A$; the state is \textbf{faithful} if $\varphi(x^*x)=0$ implies $x=0$, while it is \textbf{tracial} if $\varphi(xy)=\varphi(yx)$ for all $x,y\in A$.  (Traces tend to be denoted by $\tau$ rather than $\varphi$.)  Finally, a state $\varphi$ on a von Neumann algebra is called \textbf{normal} if the restriction of $\varphi$ to the operator norm unit ball is SOT continuous.\footnote{More generally, a linear map $M\to N$ between von Neumann algebras is normal if the restriction to the operator norm unit ball of $M$ is SOT-continuous.}  A particularly important state on $B(H)$ is that given by the map $x\mapsto \langle x\xi,\xi\rangle$ for a unit vector $\xi$ in $H$; this kind of state is called a \textbf{vector state} and is both faithful and normal.

Suppose that $\varphi$ is a state on a \cstar-algebra $A$.  One can define a pre-inner product $\langle\cdot,\cdot\rangle_\varphi$ on $A$ by $\langle x,y\rangle_\varphi:=\varphi(y^*x)$.  The quotient of $A$ by the subspace of elements $x\in A$ for which $\varphi(x^*x)=0$ is then an inner product space with the induced inner product, which we denote by $A_\varphi$.  The completion $H_\varphi$ of $A_\varphi$ is a Hilbert space for which there is a $*$-algebra homomorphism $\pi_\varphi:A\to B(H_\varphi)$ whose action on the dense subspace $A_\varphi$ is given by $\pi_\varphi(a)(\hat b):=\widehat{ab}$, where $\hat{b}$ denotes the coset of $b$ in $A_\varphi$.  This representation of $A$ is called the \textbf{GNS\footnote{GNS stands for Gelfand, Naimark, and Segal.} representation of $A$ with respect to $\varphi$}.  If $\varphi$ is faithful, then so is $\pi_\varphi$.  Note that $\varphi(a)=\langle \pi_\varphi(a)(\hat 1),\hat 1\rangle_\varphi$.  We will have occasion to consider the SOT-closure of $\pi_\varphi(A)$ in $B(H_\varphi)$, which is a von Neumann algebra which we call the \textbf{GNS closure of $A$}.  If $\varphi$ was a tracial state on $A$, then the restriction of the vector state corresponding to the identity of $A$ induces a trace on the GNS closure of $A$.    

The class of \cstar-algebras forms an elementary class in a suitable language in continuous logic \cite{MTOA2} and the corresponding ultraproduct coincides with the usual \cstar-algebraic ultraproduct.  We note that \cstar-algebras are viewed as a many-sorted structures, one sort for each operator norm ball of positive integer radius, and the metric on each sort is that induced by the operator norm.  One can extend this language by an additional unary predicate and extend the axioms of \cstar-algebras to declare that this new unary predicate is a state on the algebra.  Similarly, one can also axiomatize the case that the new predicate is a tracial state.  

The class of \textbf{tracial von Neumann algebras} is also an elementary class in an appropriate language \cite{MTOA2} and the corresonding ultraproduct corresponds to the so-called tracial ultraproduct of tracial von Neumann algebras (which is the ultraproduct referred to in the introduction in connection with the Connes Embedding Problem).  Here, a tracial von Neumann algebra is a pair $(M,\tau)$, where $M$ is a von Neumann algebra and $\tau$ is a faithful, normal trace on $M$.  We note that a tracial von Neumann algebra is viewed as a many-sorted structure, again with sorts given by operator norm balls of positive integer radii, but this time the metric on each sort is that induced by the trace, that is $d(x,y):=\|x-y\|_2$, where $\|z\|_2:=\sqrt{\tau(z^*z)}$.  By an embedding of tracial von Neumann algebras, we mean an injective, normal, trace-preserving $*$-homomorphism.

If $N\subseteq M$ are von Neumann algebras, the \textbf{relative commutant of $N$ in $M$}, denoted $N'\cap M$, is the set of elements in $M$ which commute with every element of $N$.  Von Neumann's \textbf{bicommutant theorem} states that a unital $*$-subalgebra $M$ of $B(H)$ is a von Neumann algebra if and only if $M=M''$, where $M'':=(M'\cap B(H))'\cap B(H)$ is the bicommutant of $M$ inside of $B(H)$.  Consequently, for a unital $*$-subalgebra $M$ of $B(H)$, the von Neumann algebra it generates can either be viewed as the SOT-closure of $M$ or the bicommutant $M''$ of $M$ (both computed inside $B(H)$).  For a von Neumann algebra $M$, $M'\cap M$ is also known as the \textbf{center} of $M$.  $M$ is called a \textbf{factor} if its center is trivial, that is, consists only of scalar multiples of the unit.  Factors are the ``irreducible'' von Neumann algebras as every von Neumann algebra can be decomposed as a direct integral of factors.

Both unital \cstar-algebras and von Neumann algebras can be divided into two categories: \textbf{infinite} and \textbf{finite}.  These adjectives do not apply to their cardinality, but whether or not the identity operator is a finite projection, where a projection is finite if it is not ``equivalent'' (to be precise:  \textbf{Murray-von Neumann equivalent}) to a proper subprojection.  A factor is finite if and only if it admits a trace, which is then necessarily unique.  A finite factor is of type I$_n$ if it is isomorphic to $M_n(\mathbb C)$, and otherwise it is said to be of type II$_1$.  We note that the class of II$_1$ factors is elementary in the language of tracial von Neumann algebras.  An infinite factor is either of type II$_\infty$, meaning that it is of the form $M\otimes B(H)$ for some II$_1$ factor $M$, or else of type III$_\lambda$ for some $\lambda\in [0,1]$, where the classification of type III factors (due to Connes) is determined by a particular dynamical system associated to the factor.

A von Neumann algebra $M$ is called \textbf{hyperfinite} if it is the SOT closure of an increasing sequence of subalgebras $M_1\subseteq M_2\subseteq \cdots$ with each $M_n$ a finite-dimensional matrix algebra.  There is a unique hyperfinite II$_1$ factor, denoted $\R$, which can be viewed as the direct limit (in the category of tracial von Neumann algebras) of the direct system of matrix algebras $M_{2^n}(\mathbb C)$, where the bonding maps are given by mapping $A\in M_{2^n}(\mathbb C)$ to $\left(\begin{matrix}A & 0\\ 0 & A\end{matrix}\right)\in M_{2^{n+1}}(\mathbb C)$.  (Note that this embedding preserves the normalized trace.)  We note that $\R$ embeds into every II$_1$ factor.  There is also a unique hyperfinite type III$_1$ factor, denoted $\R_\infty$, as well as, for each $\lambda\in (0,1)$, a unique hyperfinite type III$_\lambda$ factor, denoted $\R_\lambda$.  (On the other hand, there are continuum many nonisomorphic hyperfinite type III$_0$ factors.)

Given a (countable, discrete) group $G$, one can consider the left-regular representation $\lambda_G$ of $G$.  To define this, let $\ell^2(G)$ denote the Hilbert space of square-summable functions $f:G\to \mathbb{C}$ equipped with the inner product $\langle f_1,f_2\rangle:=\sum_{g\in G}f_1(g)\overline{f_2(g)}$.  We let $\delta_g\in \ell^2(G)$ denote the characteristic function of $\{g\}$; the $\delta_g$'s form an orthonormal basis for $\ell^2(G)$. For each $g\in G$, we have the unitary operator $u_g$ on $\ell^2(G)$ determined by $u_g(\delta_h):=\delta_{gh}$.  The unitary representation $\lambda_G:G\to U(\ell^2(G))$ is defined by $\lambda_G(g):=u_g$ and extends to a $*$-algebra homomorphism $\lambda_G:\mathbb{C}[G]\to \mathcal{B}(\ell^2(G))$ by linearity.  The operator norm closure of $\lambda_G(\bb C[G])$ is the \textbf{reduced group \cstar-algebra of $G$}, denoted $C^*_r(G)$, while the SOT closure of $\lambda_G(\bb C[G])$ is the \textbf{group von Neumann algebra of $G$}, denoted $L(G)$.\footnote{It is the group von Neumann algebra $L(\bb F_2)$ associated to the free group on two generators that Connes refers to in the quote from the introduction.}  One can also consider the \textbf{universal group \cstar-algebra of $G$}, denoted $C^*(G)$, defined by the universal property that any unitary representation $\pi:G\to U(H)$ extends to a $*$-homomorphism $\bar \pi:C^*(G)\to B(H)$.  The canonical map $C^*(G)\to C^*_r(G)$ is an isomorphism if and only if $G$ is amenable.  

\section{Applications to von Neumann algebras}

In this section, we explain how to use $\mre$ to give a G\"odelian refutation of the CEP as well as a handful of other undecidability results about von Neumann algebras.  We begin by carefully defining the CEP and explaining its connection with model theory.

\subsection{CEP and model theory}

The \textbf{Connes Embedding Problem (CEP)} is the statement that every tracial von Neumann algebra embeds into an ultrapower of the hyperfinite II$_1$ factor $\R$.  A basic result in (classical, discrete) model theory is that, given two structures $M$ and $N$ (in the same language), one has that $M$ embeds into an ultrapower of $N$ if and only if $M$ is a model of $\Th_\forall(N)$, the universal theory of $N$.\footnote{If $M$ is a countable model of $\Th_\forall(N)$, then any ultrapower of $N$ with respect to a nonprincipal ultrapower on $\bb N$ will contain a copy of $M$.  Otherwise, one needs to use an ultrapower with respect to a so-called \emph{good} ultrafilter.}  The corresponding result holds also in continuous logic provided that one defines the universal theory $\Th_\forall(N)$ of $N$ to be the set of non-negative universal sentences $\sigma$ for which $\sigma^N=0$\footnote{Writing $\sigma=\sup_x \theta(x)$, we see that $\sigma^N=0$ if and only if $\theta(a)=0$ for all $a\in N$, whence the ``condition'' $\sigma=0$ really is a universal statement.}.  Consequently, the CEP is the statement that every tracial von Neumann algebra is a model of the universal theory of $\R$.  Since $\R$ embeds into every II$_1$ factor and every tracial von Neumann algebra embeds into a II$_1$ factor, yet another formulation of the CEP is that all II$_1$ factors have the same universal theory, namely $\Th_\forall(\R)$.

In the sequel, it will behoove us to also consider a different perspective on the universal theory $\Th_\forall(M)$ of a metric structure $M$, namely as the function $\sigma\mapsto \sigma^M$ from the set of universal sentences (in the language of $M$) to the set of real numbers.  It is straightforward to see that these two perspectives on $\Th_\forall(M)$ carry the same information.

\subsection{CEP and computability theory}

A straightforward observation is that if $M$ is a \emph{computably presented}\footnote{We do not offer a precise definition here, but the point is that there is an enumeration of a dense subset of the structure on which the predicates applied to terms restrict to computable functions.} metric structure, then for any universal sentence $\sigma$ in the language for $M$, the value $\sigma^M$ is left c.e. uniformly in $\sigma$, in which case we say that $\Th_\forall(M)$ is \textbf{left c.e.}\footnote{Implicit here is the functional interpretation of the universal theory of a structure.}  Since $\R$ is computably presentable \cite[Remark 3.11]{hyper}, $\Th_\forall(\R)$ is left c.e.  It is natural to wonder if $\Th_\forall(\R)$ is actually computable, that is, if $\sigma^{\R}$ is computable uniformly in $\sigma$.

In \cite{compR}, with Hart, we connected this question to CEP:

\begin{thm}\label{CEPcomp}
If CEP has a positive solution, then $\Th_\forall(\R)$ is computable.
\end{thm}

The proof of Theorem \ref{CEPcomp} follows from a version of G\"odel's Completeness Theorem for continuous logic due to Ben Yaacov and Pederson \cite{BYP}.  Below, an $L$-theory $T$ is simply a set of $L$-sentences and the notation $M\models T$ means $M$ is an $L$-structure for which $\sigma^M=0$\footnote{The choice of $0$ here is simply a convenient default value.  Given the fact that any continuous function is a connective, one can recover other values of sentences using zerosets.} for all $\sigma\in T$.

\begin{thm}[Completeness theorem for continuous logic]\label{completeness}
Suppose that $T$ is an $L$-theory.  Then for any $L$-sentence $\sigma$, we have
$$\sup\{\sigma^M \ : \ M\models T\}=\inf\{r\in \mathbb{Q}^{>0} \ : \ T \vdash \sigma \dotminus r\}.$$
Here, $\vdash$ denotes the provability relation for a particular proof system for continuous logic.
\end{thm}

In connection with the previous theorem, we note that, as for any reasonable proof system, if $T$ is a c.e. set of $L$-sentences, then the set of sentences derivable from $T$ is also c.e.  

To prove Theorem \ref{CEPcomp}, it suffices to show that $\sigma^\R$ is right c.e. uniformly in $\sigma$.  Note that if CEP holds, then $\sup\{\sigma^M \ : \ M\models T_{II_1}\}=\sigma^\R$, where $T_{II_1}$ is the set of axioms for being a II$_1$ factor.  Consequently, $\sigma^\R=\inf\{r\in \mathbb{Q}^{>0} \ : \ T_{II_1}\vdash \sigma\dotminus r\}$.  Since $T_{II_1}$ is a c.e. set of sentences, the set we are taking the infimum over is c.e., whence it follows that $\sigma^\R$ is right c.e. uniformly in $\sigma$, finishing the proof of Theorem \ref{CEPcomp}.

Note that the above proof yields an immediate strengthening of Theorem \ref{CEPcomp}, to be used in the next section:

\begin{cor}\label{godel1}
Suppose there exists a c.e. set $T$ of sentences in the language of tracial von Neumann algebras for which:
\begin{enumerate}
    \item $\R\models T$.
    \item If $M\models T$, then $M$ embeds into $\R^\u$.
\end{enumerate}
Then $\Th_\forall(\R)$ is computable.
\end{cor}

\subsection{A G\"odelian refutation of CEP}

Given Theorem \ref{CEPcomp}, in order to refute CEP, it suffices to show that $\Th_\forall(\R)$ is not computable.  Of course, we will use Theorme \ref{mip}; but how?  We somehow need a way to view the quantum value of a nonlocal game as the interpretation of a universal sentence in $\R$.  It turns out that this is indeed possible once we make a slight change of setting to that of synchronous quantum strategies.

Given a correlation $p\in [0,1]^{k^2n^2}$, we say that $p$ is \textbf{synchronous} if, for all $x\in [k]$ and \emph{distinct} $a,b\in [n]$, we have $p(a,b|x,x)=0$.  In the language of nonlocal games, a correlation is synchronous if the probability that Alice and Bob answer the same question with different answers is $0$.  Given $t\in \{q,qa,qc\}$, we let $C_t^s(k,n)$ denote the elements of $C_t(k,n)$ that are synchronous.  

The following result of Paulsen et. al. \cite[Corollary 5.6]{quantum} shows that one can capture synchronous strategies using operator algebraic notions.  First, given a \cstar-algebra $A$, a \textbf{projection valued measure (PVM) of length $n$ in $A$} is a tuple $(e_1,\ldots,e_n)$ of projections\footnote{A projection in a \cstar-algebra is an element $e$ for which $e^*=e^2=e$.  In the concrete picture, they correspond to orthogonal projection operators onto closed subspaces of the ambient Hilbert space.} from $A$ such that $\sum_{a=1}^n e_a=1$.\footnote{Note that projections are positive elements, whence any PVM is a POVM once the algebra has been concretely represented}.  We let $X(n,A)$ denote the set of tuples $(e_1,\ldots,e_n)\in A^n$ that are PVMs in $A$.

\begin{thm}\label{oasync}
Let $p\in [0,1]^{k^2n^2}$.  Then $p\in C_{qc}^s(k,n)$ if and only if there is a \cstar-algebra $A$ equipped with a tracial state $\tau$ and $(e^1,\ldots,e^k)\in X(n,A)^k$ such that, writing $e^x=(e^x_1,\ldots,e^x_n)$, we have $p(a,b|x,y)=\tau(e^x_ae^y_b)$ for all $x,y\in [k]$ and $a,b\in [n]$.  Moreover, $p\in C_q(k,n)$ if and only if, in the previous sentence, $A$ can be taken to be finite-dimensional.
\end{thm}

Let $\varphi_n(x_1,\ldots,x_n)$ be the following formula in the language of tracial von Neumann algebras:  
$$\max\left(\max_{1\leq i\leq n}(\max(d(x_i^*,x_i),d(x_i^2,x_i))),\left| \sum_{i=1}^n x_i-1\right|\right).$$  Note that the zeroset of $\varphi_n$ in any tracial von Neumann algebra $M$ is the set $X(n,M)$.  The following lemma of Kim, Paulsen, and Schafhauser \cite[Lemma 3.5]{KPS} is crucial to what follows:

\begin{lem}\label{PVMdefinability}
For each $n\geq 1$ and $\epsilon>0$, there is $\delta=\delta(n,\epsilon)>0$ such that, for all $d\geq 2$ and all tuples $\vec a=(a_1,\ldots,a_n)\in M_d(\bb C)^n$ whose entries are positive contractions\footnote{A contraction is simply an element of operator norm at most $1$.} and for which $\varphi^{M_d(\bb C)}(\vec a)<\delta$, there is $(b_1,\ldots,b_n)\in X(n,M_d(\bb C))$ such that $d(a_i,b_i)<\epsilon$ for all $i=1,\ldots,n$.
\end{lem}

Note that it follows immediately that the conclusion of the previous lemma also holds with $M_k(\bb C)$ replaced by $\R$.  In model-theoretic terminology, this says that each $X(n,\R)$ is a \textbf{definable} subset of $\R^n$.  (For more on definability in continuous logic, which is a more subtle notion than its classical counterpart, we refer the reader to \cite{spectral}.)  Another expression of the definability of $X(n,\R)$ is that $X(n,\R^\u)=X(n,\R)^\u$ for all $n\geq 2$ and all ultrafilters $\u$.

We now consider the set of ``$\R$-correlations'' defined by $$C_\R(k,n):=\{p\in [0,1]^{k^2n^2} \ : \ p(a,b|x,y)=\tau_\R(e^x_ae^y_b) \text{ for some }(e^1,\ldots,e^k)\in X(n,\R)^k\}$$ and define $C_{\R^\u}(k,n)$ in the analogous manner.  An immediate consequence of the above discussion is that $C_{\R^\u}(k,n)=\overline{C_\R(k,n)}$.  This common set of correlations also coincides with the closure of the set of synchronous quantum correlations: 

\begin{lem}\label{Rlemma}
For all $k,n\geq 2$, we have $\overline{C_q^s(k,n)}=\overline{C_\R(k,n)}$.
\end{lem}

\begin{proof}
If $p\in C_q^s(k,n)$, then there is a finite-dimensional \cstar-algebra $A$ equipped with a tracial state $\tau_A$ and $(e^1,\ldots,e^k)\in X(n,A)^k$ such that $p(a,b|x,y)=\tau_A(e^x_ae^y_b)$.  By passing to the GNS closure of $A$, we may assume that $A$ is actually a von Neumann algebra.  However, any finite-dimensional tracial von Neumann algebra tracially embeds in $\R$, whence $p\in C_\R(k,n)$.  The other inclusion follows immediately from Lemma \ref{PVMdefinability}.  To see this, take $p\in C_\R(k,n)$; we seek $p'\in C_q^s(k,n)$ such that $\|p-p'\|_\infty<\epsilon$.  Take $(e^1,\ldots,e^k)\in X(n,\R)^k$ witnessing that $p\in C_\R(k,n)$.  Fix $\eta>0$ sufficiently small.  Take $d\geq 2$ sufficiently large so that there are positive contractions $f^x_a\in M_d(\bb C)$ satisfying $\|e^x_a-f^x_a\|_2<\eta$ for all $x\in [k]$ and $a\in [n]$.  If $\eta$ is sufficiently small so that $\varphi_n(f^x_1,\ldots,f^x_n)<\delta(\frac{\epsilon}{4})$ for all $x\in [k]$, it follows that there is $(g^1,\ldots,g^k)\in X(n,M_d(\bb C))^k$ such that $\|f^x_a-g^x_a\|_2<\frac{\epsilon}{4}$ for all $x\in [k]$ and $a\in [n]$.  Setting $p'(a,b|x,y):=\tau(g^x_ag^y_b)$, we have that $p'\in C_q^s(k,n)$ and $\|p-p'\|_\infty<2\eta+\frac{\epsilon}{2}<\epsilon$ if $\eta$ is also chosen to satisfy $\eta<\frac{\epsilon}{4}$.  
\end{proof}

Given an element $p\in C_{qa}^s(k,n)$, one can approximate $p$ by elements of $C_q(k,n)$; it is not a priori clear that these approximating quantum correlations can themselves be taken to be synchronous.  Thankfully, this is the case:

\begin{thm}
For each $k,n\geq 2$, we have $C^s_{qa}(k,n)=\overline{C^s_q(k,n)}$.
\end{thm}

The previous theorem was originally proven by Kim, Paulsen, and Schafhauser \cite[Theorem 3.6]{KPS} using Kirchberg's theory of \emph{amenable traces}, which was an integral part of his work connecting the CEP and the QWEP for \cstar-algebras \cite{K} (see Subection \ref{QWEPsection} below).  A more elementary proof of the previous theorem was given independently by Paddock \cite{paddock} and Vidick \cite{vidick}.

At this point, we have that the following correlation sets all coincide:
$$C^s_{qa}(k,n)=\overline{C^s_q(k,n)}=\overline{C_\R(k,n)}=C_{\R^\u}(k,n).$$

Given a nonlocal game $\frak G$ with $k$ questions and $n$ answers, we let the \textbf{synchronous quantum value of $\frak G$} be given by 
$$\val^{*,s}(\frak G):=\sup_{p\in C_{qa}^s(k,n)} \val(\frak G,p),$$ where in the above definition we could have taken the supremum instead over $C_q^s(k,n)$ or $C_\R(k,n)$ without changing the definition.  By definition, we have $\val^{*,s}(\frak G)\leq \val^*(\frak G)$ for all nonlocal games $\frak G$.

A crucial fact, a byproduct of the proof of Theorem \ref{mip}, is the following:

\begin{thm}\label{mipsync}
Theorem \ref{mip} remains true with $\val^*$ replaced by $\val^{*,s}$.
\end{thm}

Given a nonlocal game $\frak G=(\pi_\frak G,D_\frak G)$ with $k$ questions and $n$ answers, we let $\vec z$ denote the tuple of variables labeled by $z^x_a$ for $x\in [k]$ and $a\in [n]$ and let $\Theta_\frak G(\vec z)$ denote the quantifier-free formula
$$\sum_{(x,y)\in [k]^2}\pi_\frak G(x,y)\sum_{(a,b)\in [n]^2}D_\frak G(x,y,a,b)\tau(z^x_az^y_b).$$

By the above discussion, we have that $\val^{*,s}(\frak G)=\sup_{\vec e\in X(n,\R)^k}\Theta_\frak G^\R(\vec e)$.

An important fact about definable sets in continuous logic is that they are the sets that one can ``quantify over.''  More precisely (and specializing to our context), since $X(n,\cal R)^k$ is a definable subset of $\cal R^{nk}$, given any $\eta>0$, there is a sentence $\Psi_{\frak G,\eta}$ such that $|\val^{*,s}(\frak G)-\Psi_{\frak G,\eta}^\R|<\eta$.  Moreover, a careful analysis of the proof of the existence of the sentences $\Psi_{\frak G,\eta}$, together with the fact that the function $\delta(n,\epsilon)$ appearing in Lemma \ref{PVMdefinability} is computable (as can be easily verified from the proof given in \cite{KPS}), yields that the assignment $(\frak G,\eta)\mapsto \Psi_{\frak G,\eta}$ is computable.  Moreover, it can also be verified that the sentences $\Psi_{\frak G,\eta}$ are \emph{universal}!  Consequently, if $\Th_\forall(\R)$ were computable, one could effectively approximate, say, $\Psi_{\frak G,\frac{1}{4}}$, to within $\frac{1}{4}$, and could thus decide whether or not $\val^{*,s}(\frak G)$ is bounded by $\frac{1}{2}$ or equal to $1$, contradicting Theorem \ref{mipsync}.

We have just established:

\begin{thm}\label{main}
$\Th_\forall(\R)$ is not computable, whence the CEP has a negative solution.
\end{thm}

In fact, keeping in mind Corollary \ref{godel1}, we have the following ``G\"odelian refutation'' of the CEP:

\begin{cor}
There does not exist a c.e. set $T$ of sentences in the language of tracial von Neumann algebras such that $\R\models T$ and such that all models of $T$ embed in an ultrapower of $\R$.
\end{cor}

In \cite{R}, for a structure $A$ in some language, we defined the \textbf{$A$EP} to be the statement that there is a c.e.set $T$ of sentences such that $A\models T$ and such that all models of $T$ embed in an ultrapower of $A$.  Thus, we have just established that the $\R$EP has a negative solution.  We will show that the $A$EP has a negative solution for a handful of other important operator algebras below.  Note that if the $A$EP has a negative solution then, in particular, the universal theory of $A$ is not right c.e.

We remark in the special case of $\R$ that one does not need $\R\models T$ but rather only that $T$ extends the theory of II$_1$ factors, whence we obtain:

\begin{cor}
There does not exist a c.e. set $T$ of sentences extending the theory of II$_1$ factors all of whose models embed in an ultrapower of $\R$.
\end{cor}

\subsection{A digression:  $\val^{co,s}(\frak G)$ is right c.e.}\label{digression}

One can use Theorem \ref{oasync} to fulfill a promise made earlier, namely by proving the synchronous version of Proposition \ref{qcrightce}:

\begin{thm}\label{synctsirelson}
The synchronous quantum comuting value $\val^{co,s}(\frak G)$ is right c.e. uniformly in $\frak G$.
\end{thm}

\begin{proof}
Suppose that $\frak G$ has $k$ questions and $n$ answers.  Let $L$ be the language of tracial \cstar-algebras expanded by new constant symbols $e^x_a$ for all $x\in [k]$ and $a\in [n]$.  For any $r\in [0,1]$, let $\Psi_{\frak G,r}$ be the $L$-sentence given by
$$\Psi_{\frak G,r}=\max\left(\max_{x\in [k]}\varphi_n( e^x_1,\ldots,e^x_n),r\dotminus \Theta_{\frak G}(\vec e)\right),$$ where $\Theta_\frak G(\vec z)$ is the formula from the previous subsection (viewed now as a formula in the language of tracial \cstar-algebras).  Theorem \ref{oasync} implies that $\val^{co,s}(\frak G)\geq r$ if and only if $T\cup\{\Psi_{\frak G,r}\}$ is satisfiable, where $T$ is the $L$-theory of \cstar-algebras equipped with a tracial state.  

Now note that if $U$ is any continuous theory in any language, then by the Completeness Theorem \ref{completeness}, $U$ is satisfiable if and only if $U\not\vdash 1\dotminus \frac{1}{2}$. Indeed, if $U$ is satisfiable, then for any $M\models U$, we have $1^M=1$, whence $\inf\{r\in \bb Q^{>0} \ : \ U\vdash 1\dotminus r\}=1$.  Conversely, if $U$ is not satisfiable, then in the left-hand side of the equation in Theorem \ref{completeness}, we are taking the supremum of the emptyset, which, regardless of one's convention, implies $U\vdash 1\dotminus r$ for all $r\in \bb Q^{>0}$.

Thus, an algorithm for obtaining a computable sequence of upper bounds for $\val^{co,s}(\frak G)$ is obtained by running, for each rational $r\in (0,1)$, proofs from the c.e. set $T\cup\{\Psi_{\frak G,r}\}$ of sentences; if one obtains $1\dotminus \frac{1}{2}$ as a theorem, then one knows that $\val^{co,s}(\frak G)<r$.  Moreover, if $\val^{co,s}(\frak G,r)<r$, then one is guaranteed to see $1\dotminus \frac{1}{2}$ as a theorem.
\end{proof}

The previous theorem suffices to show that Tsirelson's problem has a negative solution.  Indeed, it shows that an a priori weaker statement, namely that $C^s_{qa}(k,n)=C^s_{qc}(k,n)$ for all $k,n\geq 2$, is false.  Indeed, by the computable presentability of $\R$, $\val^{*,s}(\frak G)$ is left-c.e. uniformly in $\frak G$; if equality held in the above statement, then by Theorem \ref{synctsirelson}, $\val^{*,s}(\frak G)$ would also be computable uniformly in $\frak G$, contradicting Theorem \ref{mipsync}.  We note that in \cite[Corollary 3.8]{KPS} it was pointed out that Tsirelson's problem is equivalent to its a priori weaker version in terms of synchronous correlations.

\subsection{A digression:  undecidability of the density of moments}

The computability of the universal theory of $\R$ can be reformulated into a decidability result about ``densities of moments'' that is a bit more palatable to an operator algebraist not interested in model theory.  In this subsection, we merely state the result, referring the reader to \cite[Section 4]{R} for the proof.

Given positive integers $n$ and $d$, we fix variables $x_1,\ldots,x_n$ and enumerate all *-monomials in the variables $x_1,\ldots,x_n$ of total degree at most $d$ by $m_1,\ldots,m_L$. (Of course, $L=L(n,d)$ depends on both $n$ and $d$.)  We consider the map $\mu_{n,d}:\R_1^n \rightarrow \bb D^L$ given by $\mu_{n,d}(\vec a)=(\tau(m_i(\vec a)) \ : \ i=1,\ldots,L)$.  (Here, $\bb D$ is the complex unit disk.)


We let $X(n,d)$ denote the range of $\mu_{n,d}$ and $X(n,d,p)$ be the image of the unit ball of $M_p(\mathbb C)$ under $\mu_{n,d}$.  Notice that $\bigcup_{p\in \bb N} X(n,d,p)$ is dense in $X(n,d)$.

\begin{thm}
The following statements are equivalent:
\begin{enumerate}
    \item The universal theory of $\R$ is computable.
    \item There is a computable function $F:\bb N^3\to \bb N$ such that, for every $n,d,k\in \bb N$, $X(n,d,F(n,d,k))$ is $\frac{1}{k}$-dense in $X(n,d)$.
\end{enumerate}
\end{thm}

\subsection{W$^*$ probability spaces}

Although the CEP is a statement about tracial von Neumann algebras, Ando, Haagerup, and Winslow \cite{AHW} (building on Kirchberg's reformulation of CEP in terms of the QWEP described in the next section) proved that it is equivalent to a statement about \emph{all} von Neumann algebras:

\begin{thm}\label{AHW}
The CEP is equivalent to the statement that every von Neumann algebra $M$ embeds \emph{with expectation} into the \emph{Ocneanu ultrapower} $\R_\infty^\u$ of the hyperfinite III$_1$ factor $\R_\infty$.  
\end{thm}

There are two undefined terms in the previous theorem that we now explain.  First, an embedding $\iota:M\hookrightarrow N$ between von Neumann algebras is said to be with \textbf{expectation} if it is normal and if there is a normal linear map $\varepsilon:N\to M$ such that $\varepsilon(1)=1$ and $\varepsilon(\iota(x_1)y\iota(x_2))=x_1\varepsilon(y)x_2$ for all $x_1,x_2\in M$ and $y\in N$.  For general von Neumann algebras, the existence of a normal embedding is a fairly weak condition and for technical reasons the existence of an expectation yields the correct notion of embedding for general von Neumann algebras.  We note that a normal, trace-preserving embedding between tracial von Neumann algebras is automatically with expectation.  Second, the Ocneanu ultrapower \cite{Ocneanu} is an ultrapower construction that works for all von Neumann algebras.  It begins by first identifying the ideal $I_\u$ that one quotients by, namely those sequences $(x_i)_{\in I}\in M^I$ such that $x_i$ strong* converges to $0$ along the ultrafilter\footnote{Strong* convergence just means that both $x_i$ and $x_i^*$ SOT converge to $0$ along the ultrafilter}.  One then defines $M_\u$ to be all those sequences $(x_i)_{i\in I}\in M^I$ such that $(x_i)I_\u,I_\u(x_i)\subseteq I_\u$.  The Ocneanu ultrapower is then defined to be the quotient $M^\u:=M_\u/I_\u$, which is then shown to be a von Neumann algebra.  It is readily verified that the Ocneanu ultrapower reduces to the tracial ultrapower in the case that the algebra admits a trace.

Theorem \ref{AHW} is not obviously a model-theoretic statement as there is currently no language suitable for studying all von Neumann algebras.  However, there is a model-theoretic context for the class of \textbf{W$^*$-probability spaces}, which are the pairs $(M,\varphi)$ with $M$ a von Neumann algebra and $\varphi$ a faithful, normal state on $M$.  It is known that a von Neumann algebra admits a faithful, normal state precisely when it is $\sigma$-finite, meaning that the identity is the supremum of a countable set of finite projections.  Note in particular that a tracial von Neumann algebra is a W$^*$-probability space.  We define an embedding between W$^*$-probability spaces $(M,\varphi)$ and $(N,\psi)$ to be an embedding $\iota:M\hookrightarrow N$ with expectation $\epsilon:N\to M$ such that $\varepsilon$ (and thus $\iota$) are state-preserving.  Also, the ultrapower of a W$^*$-probability space $(M,\varphi)$ is defined to be $(M^\u,\varphi^\u)$, where $M^\u$ is the Ocneanu ultrapower of $M$ and $\varphi^\u$ is the ultrapower state on $M^\u$.  A first-order axiomatization for W$^*$-probability spaces recovering the Ocneanu ultraproduct is due to Dabrowski \cite{dabrowski}.  A much simpler axiomatization of W$^*$-probability spaces will be presented in the forthcoming article \cite{wstar}.

Now suppose that $(M,\varphi)$ is a W$^*$-probability space and $\iota:M\hookrightarrow \R_\infty^\u$ is an embedding with expectation $\varepsilon:\R^\infty_\u \to M$.  One can define a faithful, normal state $\psi$ on $\R^\infty_\u$ by setting $\psi(x):=\varphi(\varepsilon(x))$, thus obtaining an embedding $(\iota,\varepsilon):(M,\varphi)\hookrightarrow (\R^\infty_\u,\psi)$ of W$^*$-probability spaces.  Consequently, we see that the CEP implies that, for each W$^*$-probability space $(M,\varphi)$, there is a faithful, normal state $\psi$ on $\R_\infty^\u$ such that $(M,\varphi)\hookrightarrow (\R_\infty^\u,\psi)$.  Conversely, if every W$^*$-probability space embeds into $(\R_\infty^\u,\psi)$ for some faithful, normal state $\psi$ on $\R_\infty^\u$, then it follows that CEP holds.\footnote{If a tracial von Neumann algebra embeds into $\R^\infty_\u$ with expectation, then it has QWEP, whence it embeds tracially into $\R^\u$; see the next section.}  This formulation of the CEP does not appear to have the same flavor as the original formulation as the choice of $\psi$ seems to depend on the state $\varphi$, whence there is no canonical W$^*$-probability space whose ultrapower contains all others.  However, a minor miracle occurs:  as shown in \cite[Theorem 4.20]{AH12}, for any two states $\psi_1,\psi_2$ on $\R_\infty^\u$, there is a unitary element $u$ in $\R_\infty^\u$ such that $u\psi_1u^*=\psi_2$, whence a W$^*$-probability space embeds into $(\R_\infty^\u,\psi_1)$ if and only if it embeds into $(\R_\infty^\u,\psi_2)$.  Consequently, given any state $\varphi_\infty$ on $\R_\infty$, whether or not a W$^*$-probability space embeds into $(\R_\infty,\varphi_\infty)^\u$ is independent on the choice of state $\varphi_\infty$.  We have thus arrived at:

\begin{thm}
The CEP is equivalent to the statement:  for any W$^*$-probability space $(M,\varphi)$ and some (equiv. any) state $\varphi_\infty$ on $\R_\infty$, we have that $(M,\varphi)$ embeds into $(\R_\infty,\varphi_\infty)^\u$, or, in model-theoretic terms, that $(M,\varphi)$ is a model of the universal theory of $(\R_\infty,\varphi_\infty)$.
\end{thm}

In \cite{QWEPundec}, we obtain a G\"odelian refutation\footnote{This is indeed a G\"odelian refutation as the axioms for being a W$^*$-probability space are c.e.} of this formulation of the CEP by showing the following:

\begin{thm}
The $(\R_\infty,\varphi_\infty)$EP has a negative solution.
\end{thm}

The basic idea of the proof stems from the fact, mentioned later on, that a tracial von Neumann algebra embeds into $\R_\infty^\u$ as a W$^*$-probability space if and only if it embeds into $\R^\u$ as a tracial von Neumann algebra, whence one could leverage a positive solution to the $(\R_\infty,\varphi_\infty)$EP to obtain a positive solution to the $\R$EP.  That being said, there are certain complications arising from the language used to describe W$^*$-probability spaces that must be dealt with. 

Since $\R_\infty$ embeds into any type III$_1$ factor \cite[Theorem 3.5]{HM}, we can strengthen the previous theorem (as we did in the case of tracial von Neumann algebras) by stating that there is no c.e. set of axioms extending the axioms for being a W$^*$-probability space which has at least one model that is a type III$_1$ factor and all of whose models embed into $\R_\infty^\u$.

Ando, Haagerup, and Winslow showed that Theorem \ref{AHW} remains true if $\R_\infty$ is replaced by $\R_\lambda$ for any $\lambda\in (0,1)$.  This time, the universal theory of $\R_\lambda$, viewed as a W$^*$-probability space, is no longer independent on the choice of state.  Nevertheless, there is a canonical state $\varphi_\lambda$ on $\R_\lambda$, known as the \textbf{Powers' state}.  In \cite{QWEPundec}, we showed:

\begin{thm}
The $(\R_\lambda,\varphi_\lambda)$EP has a negative solution.
\end{thm}

This time, the key idea is that, for certain kinds of faithful, normal states known as \textbf{lacunary states} (of which the Powers' states $\varphi_\lambda$ are an example), the \textbf{centralizer} $M_\varphi$ of $M$, defined to be those $x\in M$ for which $\varphi(xy)=\varphi(yx)$ for all $y\in M$, is an effectively definable subset of $M$.  Since the centralizer of $\varphi_\lambda$ is $\R$, one can leverage a positive solution to the $(\R_\lambda,\varphi_\lambda)$EP to obtain a positive solution to the $\R$EP.  Once again, there are a number of technicalities to be dealt with and the interested reader can consult \cite{QWEPundec}.

\section{Applications to \cstar-algebras}

In this section, we show how the ideas from the previous section can be used to establish a number of incompleteness results for \cstar-algebras.  For the sake of simplicity, in this section, all \cstar-algebras are assumed to be unital and all $*$-homomorphisms are assumed to be unit-preserving.

\subsection{Passing to the GNS closure}\label{GNSsection}

In this subsection, we show how to use the fact that the $\R$EP has a negative solution to show that the $A$EP has a negative solution for certain \cstar-algebras $A$.  

Suppose that $A$ is a \cstar-algebra for which the $A$EP has a positive solution, say $T_0$ is a c.e. set of sentences such that $A\models T_0$ and all models of $T_0$ embed into $A^\u$.  In order to connect to the $\R$EP, we assume that $A$ possesses a tracial state $\tau_A$.  

In the remainder of this section, $\sigma$ denotes a universal sentence in the language of tracial von Neumann algebras.  In order to obtain a contradiction, we wish to show the following two things:
\begin{enumerate}
    \item $\sigma^{(A,\tau_A)}=\sigma^{(\R,\tau_\R)}$, where on the left hand side of the equation, we view $\sigma$ as a sentence in the language of tracial \cstar-algebras.
    \item There is a c.e. set $T\supseteq T_0$ in the language of tracial \cstar-algebras such that $\sigma^{(A,\tau_A)}=\sup\{\sigma^{(M,\tau)} \ : \ (M,\tau)\models T\}$.
\end{enumerate}

If we are successful, then by item (2) and the Completeness Theorem, $\sigma^{(A,\tau_A)}$ is right c.e. uniformly in $\sigma$, and thus so is $\sigma^{(\R,\tau_R)}$ by (1), contradicting Theorem \ref{main}.  

To achieve (1), we first let $(N,\tau_N)$ denote the GNS closure of $(A,\tau_A)$.  We note that if $A$ is simple, then $\tau_A$ is faithful and so $A$ embeds into $N$, whence $\sigma^{(A,\tau_A)}=\sigma^{(N,\tau_N)}$.  (Simplicity gives that $\sigma^{(A,\tau_A)}\leq\sigma^{(N,\tau_N)}$, while equality follows from the fact that $A$ is $\|\cdot\|_2$-dense in $N$.)   In order to finish the proof of (1), we wish to show that $\sigma^{(N,\tau_N)}=\sigma^{(\R,\tau_R)}$.  \emph{If} $N$ is a II$_1$ factor, then $\R$ embeds into $N$ and thus $\sigma^{(N,\tau_N)}\geq \sigma^{(\R,\tau_R)}$.  To guarantee that $N$ is a factor, it suffices to ensure that $N$ has a unique trace, which would follow if $A$ was assumed to have a unique trace (otherwise known as being \textbf{monotracial}).  If moreover $A$ is infinite-dimensional, then so is $N$, and thus $N$ is a II$_1$ factor.  The other inequality $\sigma^{(N,\tau_N)}\leq \sigma^{(\R,\tau_R)}$ would follow \emph{if} we assumed that $N$ embeds into $\R^\u$.

To summarize the discussion thus far:  if $A$ is a simple, infinite-dimensional, monotracial \cstar-algebra whose GNS closure $N$ embeds into $\R^\u$, then (1) is satisfied.

How can we ensure that (2) holds?  To begin with, we should assume that $T$ contains $T_0$ and the axioms for being a tracial \cstar-algebra.  We also require that $(A,\tau_A)\models T$ so that the left-hand side of (2) is at most the right-hand side of (2).  To achieve the other inequality, suppose that $(B,\tau)\models T$; we wish to show that $\sigma^{(B,\tau)}\leq \sigma^{(A,\tau_A)}$.  Since $B\models T_0$, we know that $B$ embeds into $A^\u$.  \emph{If} this embedding was trace-preserving (where $A^\u$ is equipped with the trace $\tau_A^\u$), then we would have the desired inequality.  The embedding would be trace-preserving if $B$ itself were monotracial.  So it would seem that, in order to finish the proof of (2), we should require that all models of $T$ be monotracial.  Unfortunately, it is currently unknown whether or not being monotracial is axiomatizable.  However, there is a first-order condition, called the \textbf{$(m,\gamma)$-uniform Dixmier property} (for some choice of integer $m$ and positive constant $\gamma$), axiomatized in fact by a single sentence $\theta_{m,\gamma}$, such that any model of $\theta_{m,\gamma}$ that admits a trace is in fact monotracial; see \cite[Section 7.2]{munster} for more details.

We have thus arrived at the following conclusion, where having the uniform Dixmier property means having the $(m,\gamma)$-uniform Dixmier property for some choice of parameters $(m,\gamma)$:

\begin{thm}
Suppose that $(A,\tau_A)$ is a simple, infinite-dimensional, tracial C*-algebra with the uniform Dixmier property whose GNS closure embeds into an ultrapower of $\cal R$.  Then the $A$EP has a negative solution.
\end{thm}

In the next section, we turn to a particular class of examples where the GNS closure embeds into $\R$ itself; in the section that follows, we will consider the condition in full generality.

\subsection{Nuclear \cstar-algebras}\label{nuclearsection}

One of the main technical issues in \cstar-algebra theory is the question of how to take the tensor product of two \cstar-algebras.  More precisely, given \cstar-algebras $A$ and $B$, one can consider their ``algebraic tensor product'' $A\odot B$, which is simply their tensor product as vector spaces.  It is straightforward to equip $A\odot B$ with a $*$-algebra structure; the question then becomes:  how should one define a \emph{\cstar-norm} on $A\odot B$, that is, a norm on $A\odot B$ whose completion is a \cstar-algebra?

There are two natural \cstar-norms on $A\odot B$.  The first, called the \textbf{minimal tensor norm}, denoted $\|\cdot\|_{\min}$, is defined by first concretely representing $A\subseteq B(H_A)$ and $B\subseteq B(H_B)$, in which case we have the concrete representation $A\odot B\subseteq B(H_A\otimes H_B)$.  The minimal norm $\|\cdot\|_{\min}$ is the restriction of the operator norm on $B(H_A\otimes H_B)$ to $A\odot B$; it can be verified that this norm is independent of the choice of representations of $A$ and $B$.  The completion of $A\odot B$ with respect to $\|\cdot\|_{\min}$ is called the \textbf{minimal tensor product} of $A$ and $B$, denoted $A\otimes_{\min}B$ or sometimes even $A\otimes B$ as it is perhaps the most commonly considered tensor product on \cstar-algebras.

The second norm on $A\odot B$ is called the \textbf{maximal tensor norm}, denoted $\|\cdot\|_{\max}$, and is defined by $\|x\|_{\max}:=\sup\{\|\pi(x)\| \ : \ \pi:A\odot B\to B(H) \text{ a $*$-homomorphism}\}$.  The completion of $A\odot B$ with respect to $\|\cdot\|_{\max}$ is called the \textbf{maximal tensor product} of $A$ and $B$, denoted $A\otimes_{\max}B$.

The minimal and maximal tensor norms get their name from the fact that, for any other \cstar-norm $\|\cdot\|_\alpha$ on $A\odot B$, we have that $\|\cdot\|_{\min}\leq \|\cdot\|_\alpha\leq \|\cdot\|_{\max}$.

In general, the norms $\|\cdot\|_{\min}$ and $\|\cdot\|_{\max}$ do not coincide.  When they do coincide (in other words, when there is a unique tensor norm on $A\odot B$), we call $(A,B)$ a \textbf{nuclear pair}.  When $(A,B)$ is a nuclear pair for all \cstar-algebras $A$, we say that $A$ is a \textbf{nuclear} \cstar-algebra.  All finite-dimensional \cstar-algebras are nuclear as are all abelian \cstar-algebras.  The reduced group \cstar-algebra $C^*_r(G)$ is nuclear if and only if $G$ is amenable.

The class of nuclear \cstar-algebras is an intensely studied class of algebras for in many circumstances they can be classified in terms of concrete invariants.  The class can alternatively be defined as those algebras that can be approximated by so-called completely positive maps into and out of matrix algebras; this description of nuclear algebras allows one to show that they form an omitting types class  \cite[Theorem 5.7.3(3)]{munster}.

The connection between nuclearity and the discussion in the previous section is the following:

\begin{fact}
If $(A,\tau)$ is a tracial nuclear \cstar-algebra, then the GNS closure of $(A,\tau)$ is hyperfinite.
\end{fact}

\begin{cor}\label{nuclearcorollary}
Suppose that $(A,\tau_A)$ is a simple, infinite-dimensional, \emph{nuclear} tracial C*-algebra with the uniform Dixmier property.  Then the $A$EP has a negative solution.
\end{cor}

One prominent example of a \cstar-algebra satisfying all of the adjectives in the preceding example is the \textbf{universal UHF algebra} $\mathcal{Q}$, which is the inductive limit (in the category of unital \cstar-algebras) of any sequence 
$$M_{k_1}(\bb C)\hookrightarrow M_{k_2}(\bb C)\hookrightarrow \cdots$$ of matrix algebras such that every positive integer divides some $k_j$.  It follows that the $\mathcal{Q}$EP has a negative solution.  A prominent open problem in \cstar-algebra theory, known as the \textbf{Blackadar-Kirchberg problem} or \textbf{MF problem}, asks if all \textbf{stably finite} \cstar-algebras embed into an ultrapower of $\cal Q$.  Here, a \cstar-algebra $A$ is stably finite if $A$ and all of its matrix amplifications $M_n(A)$ are finite.  (Note that this is a necessary condition for admitting an embedding into $\cal Q^\u$.)  It is straightforward to see that a negative resolution to the CEP implies  a negative resolution of the MF problem; see \cite[Proposition 6.1]{R}.  However, the negative resolution to the $\cal Q$EP can be viewed as a G\"odelian refutation of the MF problem, for being stably finite is effectively axiomatizable, and the $\cal Q$EP shows that no c.e. extension of these axioms can ensure embedability into $\cal Q^\u$.

Another prominent algebra which satisfies the hypotheses of Corollary \ref{nuclearcorollary} is the \textbf{Jiang-Su} algebra $\cal Z$, whch plays a very important role in the classification program for nuclear \cstar-algebras.  A natural analog of the MF problem involving $\cal Q$ would ask that all \textbf{stably projectionless} \cstar-algebras embed into an ultrapower of $\cal Z$.  (See \cite[Remark 6.7]{R} for a precise definition of stably projectionless.)  Since being stably projectionless is an effectively axiomatizable property, a consequence of the negative solution of the $\cal Z$EP is the fact that not all stably projecitonless \cstar-algebras embed into an ultrapower of $\cal Z$.  At the time of writing, a proof of this fact avoiding model-theoretic techniques has not been found. 

\subsection{QWEP \cstar-algebras}\label{QWEPsection}

Given \cstar-algebras $A\subseteq B$ and $C$, while there is always an isometric inclusion $A\otimes_{\min}C\subseteq B\otimes_{\min}C$, the natural map $A\otimes_{\max}C\to B\otimes_{\max} C$ need not always be an isometric inclusion.  Following Lance \cite{lance}, we say that $A$ has the \textbf{weak expectation property (WEP)} if, for any $B\supseteq A$ and any $C$, the natural map $A\otimes_{\max}C\to B\otimes_{\max}C$ is an isometric inclusion.

Kirchberg \cite{K} proved that $A$ has the WEP if and only if $(A,C^*(\bb F_\infty))$ is a nuclear pair, where $\bb F_\infty$ is the nonabelian free group on a countably infinite set of generators.

It thus becomes natural to ask:  does $C^*(\mathbb{F}_\infty)$ itself have the WEP?  Equivalently, is it true that $(C^*(\mathbb{F}_\infty),C^*(\mathbb{F}_\infty))$ is a nuclear pair?  (The analogous statement for $C^*_r(\mathbb{F}_\infty)$ is false.)

If $C^*(\mathbb{F}_\infty)$ has the WEP, then every separable \cstar-algebra is a quotient of a \cstar-algebra with the WEP; we abbreviate the property of being a quotient of a \cstar-algebra with the WEP by saying that the algebra has the \textbf{QWEP}.  Conversely, if every separable \cstar-algebra has the QWEP, then in particular $C^*(\mathbb{F}_\infty)$ has the QWEP.  For \cstar-algebras with the so-called \textbf{lifting property} (such as $ C^*(\bb F_\infty)$), having the QWEP is equivalent to having the WEP.  Moreover, it can be checked that all separable \cstar-algebras have the QWEP if and only if all \cstar-algebras have the QWEP.  Consequently, we see that $C^*(\mathbb{F}_\infty)$ has the WEP if and only if all \cstar-algebra have the QWEP.  A deep theorem of Kirchberg \cite{K} is that these two equivalent statements are in turn equivalent to a positive solution to the CEP:

\begin{thm}
The following are equivalent:
\begin{enumerate}
\item $C^*(\mathbb F_\infty)$ has the WEP.
\item Every \cstar-algebra has the QWEP.
\item CEP has a positive answer.
\end{enumerate}
\end{thm}

A key ingredient in the proof of the previous theorem is the following result of Kirchberg \cite{K}:

\begin{thm}\label{Kirchberg}
A tracial von Neumann algebra has the QWEP if and only if it (tracially) embeds in $\R^\u$.
\end{thm}

\begin{remark}
The main result of Ando, Haagerup, and Winslow discussed in the previous section is a generalization of the Theorem \ref{Kirchberg}:  a von Neumann algebra has QWEP if and only if it embeds into $\R_\infty^\u$ (equiv. $\R_\lambda^\u$) with expectation.
\end{remark}

\begin{remark}
It follows from Theorem \ref{Kirchberg} that for a tracial \cstar-algebra $(A,\tau_A)$, we have that $A$ has the QWEP if and only if the GNS closure of $A$ tracially embeds in $\R^\u$, thus clarifying the condition from the previous section.
\end{remark}

Although the class of \cstar-algebras with the WEP is not elementary (it is not closed under ultraproducts \cite[Corollary 4.14]{GoldSinc}), we observed in \cite{GoldQWEP} that the class of \cstar-algebras with the QWEP is elementary.  This result was obtained by showing that the class of \cstar-algebras with the QWEP is closed under ultraproducts and ultraroots.  Moreover, since the class of \cstar-algebras with the QWEP is closed under direct limits, the class must be axiomatized by $\forall\exists$-axioms.  However, ``concrete'' axioms for the class were not offered in \cite{GoldQWEP} and it is reasonable to wonder if there is an effective set of axioms for this class.  By the above theorem, if the CEP were true, then the c.e. set of axioms for \cstar-algebras would (trivially) axiomatize the class.  Once again, we can give a G\"odelian refutation, as shown in \cite{QWEPundec}:

\begin{thm}\label{AGH}
There is no c.e. set $T$ of sentences in the language of \cstar-algebras with the following two properties:
\begin{enumerate}
\item All models of $T$ have QWEP.
\item There is an infinite-dimensional, monotracial model $A$ of $T$ whose unique trace is faithful.
\end{enumerate}
In particular, there is no effective theory $T$ in the language of \cstar-algebras that axiomatizes the QWEP \cstar-algebras.
\end{thm}

The proof of the previous theorem is very similar to the proof from Subection \ref{GNSsection} above.  Indeed, suppose, towards a contradiction, that such a set $T$ existed.  Take an infinite-dimensional, monotracial model $A$ of $T$ whose unique trace $\tau_A$ is faithful.
Work now in the language of tracial \cstar-algebras and consider the theory $T'$ consisting of the axioms for tracial \cstar-algebras together with $T$.  Note that $T'$ is effective and $(A,\tau_{A})\models T'$.  
Note that, for any universal sentence $\sigma$ in the language of tracial von Neumann algebras, we have
$$\sup\{\sigma^{(B,\tau_B)} \ : \ (B,\tau_B)\models T'\}=\sigma^{(\mathcal{R},\tau_{\mathcal{R}})}.$$  The proof of this statement follows the same line of reasoning as earlier, namely that the GNS closure of $(A,\tau_A)$ is a II$_1$ factor and all models of $T$, being QWEP, have GNS closures that embed into $\R^\u$.  By running proofs from $T'$, we can find computable upper bounds to $\sigma^{(\mathcal{R},\tau_{\mathcal{R}})}$, obtaining a contradiction as before.

\subsection{Tsirelson pairs of \cstar-algebras}

The following theorem, due independently to Fritz \cite{fritz} and Junge et. al. \cite{junge} is the main observation connecting the CEP and Tsirelson's problem.  Here, $\mathbb F(k,n)$ is the group freely generated by $k$ elements of order $n$.

\begin{thm}\label{fritz}
Given  $p\in [0,1]^{k^2n^2}$, we have that $p\in C_{qa}(k,n)$ (resp. $p\in C_{qc}(k,n)$) if and only if, for each $x\in [k]$ there are POVMs $A^x$ and $B^y$ of length $n$ in \cstar$(\mathbb{F}(k,n))$ and a state $\varphi$ on \cstar$(\mathbb{F}(k,n))\otimes_{\min}$ \cstar$(\mathbb{F}(k,n))$ (resp. on \cstar$(\mathbb{F}(k,n))\otimes_{\max}$\cstar$(\mathbb{F}(k,n))$ such that $$p(a,b|x,y)=\varphi(A^x_a\otimes B^y_b).$$
\end{thm}

Note that if $(C^*(\mathbb{F}(k,n)),C^*(\mathbb{F}(k,n)))$ is a nuclear pair, then Tsirelson's problem would have a positive solution.  It turns out that after a small ``lifting'' argument, $(C^*(\mathbb{F}(k,n)),C^*(\mathbb{F}(k,n)))$ being a nuclear pair would follow from $(C^*(\mathbb{F}_\infty),C^*(\mathbb{F}_\infty))$ being a nuclear pair.  Combined with Theorem \ref{Kirchberg}, we have arrived at:

\begin{cor}
A positive solution to the CEP implies a positive solution to Tsirelson's problem.
\end{cor}

Motivated by Theorem \ref{fritz}, with Hart \cite{tsirelson} we considered the following definitions:

\begin{defn}
Let $C_{\min}(C,D,k,n)$ (respectively $C_{\max}(C,D,k,n))$ denote the closure of the set of correlations of the form $\varphi(A^x_a\otimes B^y_b)$, where $A^1,\ldots,A^k$ are POVMs of length $n$ from $C$, $B^1,\ldots,B^k$ are POVMs of length $n$ from $D$, and $\varphi$ is a state on $C\otimes_{\min} D$ (respectively a state on $C\otimes_{\max} D$).
\end{defn}

The remainder of this section describes the main results of \cite{tsirelson}.  Throughout, $(C,D)$ denotes a pair of \cstar-algebras.

We begin by noting the following facts about the previous definition:

\begin{lem}
    
\

\begin{enumerate}
\item $C_{\min}(C,D,k,n)\subseteq C_{\max}(C,D,k,n)$.
    \item $C_{\min}(C,D,k,n)\subseteq C_{qa}(k,n)$.
    \item $C_{\max}(k,n)\subseteq C_{qc}(k,n)$.
\end{enumerate}
\end{lem}
\begin{defn}
We say that $(C,D)$ is a \textbf{(strong) Tsirelson pair} if $C_{\min}(C,D,k,n)=C_{\max}(C,D,k,n)$ (=$C_{qa}(k,n)$) for all $(k,n)$.
\end{defn}

Clearly, if $(C,D)$ is a nuclear pair, then $(C,D)$ is a Tsirelson pair.

Tsirelson's problem asks if $(C^*(\mathbb{F}_\infty),C^*(\mathbb{F}_\infty))$ is a Tsirelson pair; we now know that this is not the case.

\begin{lem}
Exactly one of the following happens:
\begin{itemize}
\item $(C,D)$ is not a Tsirelson pair.
\item One of $C$ or $D$ is \emph{subhomogeneous} (whence $(C,D)$ is a nuclear pair), but $(C,D)$ is not a strong Tsirelson pair.
\item $(C,D)$ is a strong Tsirelson pair.
\end{itemize}
\end{lem}

Here, a \cstar-algebra $A$ is \textbf{subhomogeneous} if there is $n\in \bb N$ such that all irreducible representations of $A$ have dimension at most $n$.  The subhomogeneous \cstar-algebras form a very small subclass of the class of nuclear \cstar-algebras.  The above lemma can thus be construed as saying that if a Tsirelson pair is not a strong Tsirelson pair, then there is a very particular reason for that being the case.

\begin{lem}
The class of Tsirelson pairs is closed under taking quotients, by which we mean, if $(C,D)$ is a Tsirelson pair and $C'$ is a quotient of $C$ and $D'$ is a quotient of $D$, then $(C',D')$ is also a Tsirelson pair.
\end{lem}

A consequence of the preceding lemma is the following result, which shows that there are many examples of (strong) Tsirelson pairs that are not nuclear pairs:

\begin{prop}
If $C$ has QWEP, then $(C,D)$ is a Tsirelson pair.
\end{prop}

In other words, QWEP \cstar-algebras have the Tsirelson property in the sense of the following definition:

\begin{defn}
$C$ has the \textbf{Tsirelson property (TP)} if $(C,D)$ is a Tsirelson pair for any \cstar-algebra $D$.
\end{defn}

We observed the following facts about \cstar-algebras with the TP:

\begin{prop}

    \
    
    \begin{enumerate}
        \item $C$ has the TP if and only if $(C,C^*(\mathbb{F}_\infty))$ is a Tsirelson pair.
        \item The class of \cstar-algebras with TP is closed under direct limits, quotients, relatively weakly injective subalgebras, and ultraproducts.  In particular, the class of \cstar-algebras with the TP is an axiomatizable class.
    \end{enumerate}  

\end{prop}

\begin{defn}
$C$ has the \textbf{strong Tsirelson property (STP)} if and only if it has the TP and is not subhomogeneous.
\end{defn}

It can be shown that $C$ has the STP if and only if $(C,D)$ is a strong Tsirelson pair for every non-subhomogeneous $D$.  Moreover, since the class of non-subhomogeneous algebras is axiomatizable \cite[Section 2.5]{munster}, it follows that the STP is also an axiomatizable property.

It is natural to ask:  are there explicit axioms for the class of \cstar-algebras with the (S)TP?  Continuing with the theme of this article, the answer is a resounding no, as the following results indicate.

\begin{thm}\label{tsirelsonundecidable}
There is no c.e. theory $T$ in the language of pairs of \cstar-algebras such that all models of $T$ are Tsirelson pairs and at least one model of $T$ is a strong Tsirelson pair.
\end{thm}

Note the requirement that at least one model of $T$ being a strong Tsirelson pair is necessary in the previous theorem, as, for example, the theory $T$ of pairs of abelian \cstar-algebras is c.e. and all models are Tsirelson pairs.

Before giving the proof of Theorem \ref{tsirelsonundecidable}, we give a couple of corollaries:

\begin{cor}
There is no c.e. theory $T$ in the language of \cstar-algebras such that all models have the TP and at least one model with the STP.
\end{cor}

The previous corollary allows us to improve Theorem \ref{AGH} from the previous subsection:

\begin{cor}
There is no c.e. theory $T$ in the language of \cstar-algebras such that all models have the QWEP and at least one model that is not subhomogeneous.
\end{cor}

\begin{proof}[Proof of Theorem \ref{tsirelsonundecidable}]  Suppose, towards a contradiction, that such a c.e. theory $T$ exists.  Expand the language by a new unary predicate symbol $P$ and let $T'$ be the c.e. extension of $T$ stating that its models are of the form $(C,D,P)$, where $P(c,d)=\varphi(c\otimes d)$ for some state $\varphi$ on $C\otimes_{\max}D$; this is possible since states on $C\otimes_{\max}D$ ``are'' just extensions of unital linear functionals on $C\odot D$ that are positive on $C\odot D$.

Given a nonlocal game $\mathfrak{G}$, we can consider the universal sentence $\sigma_{\mathfrak{G}}$ in this extended language given by $$\sup_{\vec A}\sup_{\vec B} \sum_{(x,y)\in [k]}\pi(x,y)\sum_{(a,b)\in [n]}D(x,y,a,b)P(A^x_a,B^y_b),$$ where the quantifications over POVMs here is legitimate (and effective) since they can be shown to form a definable set (relative to the theory $T'$).  The assumptions on the theory show that  $$\sup\{\sigma_{\mathfrak{G}}^{(C,D,P)} \ : \ (C,D,P)\models T'\}=\operatorname{val}^*(\mathfrak{G}).$$
The inequality $\leq$ uses that all models of $T$ are Tsirelson pairs while the inequality $\geq$ uses that at least one model is a strong Tsirelson pair.  Running proofs from $T'$, we get computable upper bounds to $\operatorname{val}^*(\mathfrak{G})$.  Since this procedure is uniform in $\frak G$, we obtain a contradiction to Theorem \ref{mip}.
\end{proof} 

\section{The Aldous-Lyons conjecture}

In this section, we describe a very recent result showing how (a variant of) Theorem \ref{mip} was used to settle a prominent conjecture in probability theory.

\subsection{Sofic groups and graphs}
A (countable, discrete) group $G$ is called \textbf{hyperlinear} if the group von Neumann algebra $L(G)$ embeds into $\R^\u$.  As shown by Radulescu \cite{radulescu}, $G$ is hyperlinear if and only if it embeds into a metric ultraproduct $\prod_\u U(n)$ of finite-dimensional unitary groups, where each unitary group is equipped with the metric associated to the normalized trace on $M_n(\bb C)$.  Although in theory it is possible to construct a somewhat explicit counterexample to CEP from the somewhat explicit correlation belonging to $C_{qc}(k,n)\setminus C_{qa}(k,n)$ (see Remark \ref{separation}), it is not at all clear that this counterexample is a group von Neumann algebra.  Thus, the existence of a non-hyperlinear group remains an open problem.

An even weaker question remains open, that is, whether or not there exists a group that is not \textbf{sofic}.  Here, a group is sofic if it embeds into a metric ultraproduct $\prod_\u S_n$ of finite symmetric groups, where each $S_n$ is equipped with its normalized Hamming metric measuring the proportion of elements on which two permutations disagree.  By associating to each permutation its associated permutation matrix, it is straightforward to see that every sofic group is hyperlinear.  Sofic groups were introduced by Gromov in \cite{gromov}, where he showed that Gottschalk’s Surjectivity Conjecture holds for the class of sofic groups; the term ``sofic'' was coined by Weiss in \cite{weiss}.  For a thorough discussion of the class of sofic groups, see the author's book \cite[Chapter 13]{ultrabook}.

The notion of soficity was later extended to the setting of graphs; the context is as follows.  A \textbf{rooted graph} is a pair $(G,o)$, where $G$ is a graph and $o$ is a vertex of $G$.  An isomorphism of rooted graphs is a root-preserving isomorphism of graphs.  Given $r>0$, we let $B_r(G,o)$ denote those vertices that lie at a distance at most $r$ from the root $o$, which is then a rooted graph in its own right.  By a \textbf{random rooted graph} we simply mean a probability space whose elements are rooted graphs.  We abuse notation and denote a random rooted graph by $(G,o)$ as if it were deterministic.  We view a finite graph as a random rooted graph by uniformly randomly selecting a vertex.  A notion of convergence for random rooted graphs was introduced by Benjamini and Schramm in \cite{BS}:  a sequence $(G_n,o_n)$ of random rooted graphs \textbf{Benjamini-Schramm converges} to the random rooted graph $(G,o)$ if, for any $r>0$ and any finite rooted graph $(H,p)$, the probability that $B_r(G_n,o_n)\cong (H,p)$ converges as $n\to \infty$ to the probability that $B_r(G,o)\cong (H,p)$.  A random rooted graph is called \textbf{sofic} if it is the Benjamini-Schramm limit of a sequence of finite graphs (viewed as random rooted graphs as mentioned above).  

One can verify that a finitely generated group is sofic if and only if one (equiv. any) of its Cayley graphs is sofic, where the Cayley graph is viewed as a deterministic rooted graph with root given by the identity of the group.

\subsection{Unimodular graphs and the Aldous-Lyons conjecture} Following Aldous and Lyons \cite{AL}, call a random rooted graph $(G,o)$ \textbf{unimodular} if, for any nonnegative function\footnote{For the displayed expectations to be defined, some sort of measurability criteria must be assumed of $f$.} $f$ defined on \emph{doubly rooted graphs} (where two vertices are distinguished), one has that
$$\bb E_{(G,o)}\left[ \sum_{x\in V(G)} f(G,o,x)\right]=\bb E_{(G,o)}\left[ \sum_{x\in V(G)} f(G,x,o)\right].$$
This is often motivated by calling $f$ a \emph{mass transport function}, where $f(G,o,x)$ denotes the amount of mass transported from $o$ to $x$, and then the unimodularity condition asserts that the expected amount of mass leaving the root is equal to the expexcted amount of mass entering the root.  The term unimodular stems from the fact that when the random rooted graph is simply a deterministic, transitive graph (meaning that the automorphism group acts transitively on its vertices) with any choice of root, then the graph is unimodular precisely when the automorphism group is a unimodular locally compact group (meaning its left and right Haar measures coincide).  Cayley graphs of finitely generated groups are transitive; it can be shown that they are indeed unimodular when viewed as deterministic rooted graphs.  (Something even more general holds; see the result of Abert, Glasner, and Virag in the next subsection.)

It is straightforward to verify that any finite graph, viewed as a random rooted graph, is unimodular.  This behavior persists in the limit, that is, every sofic random rooted graph is unimodular.  The \textbf{Aldous-Lyons conjecture} states that the converse is true, namely that the classes of sofic and unimodular random rooted graphs coincide.  Since the Cayley graph of a finitely generated group is unimodular, the Aldous-Lyons conjecture would imply that the graph is sofic, which, as mentioned above, would imply that the group itself is sofic.  In other words, the Aldous-Lyons conjecture implies that all finitely generated (and thus all) groups are sofic!

\subsection{Invariant random subgroups and an algebraic consequence of the Aldous-Lyons conjecture} Fix a group $G$ and let $\sub(G)$ denote the set of subgroups of $G$, viewed as a subset of $2^G$ by identifying a subgroup with its characteristic function; as such, $\sub(G)$ is a closed (and thus compact) subset of $2^G$.  A \textbf{random subgroup of $G$} is a Borel probability measure on $\sub(G)$.  Given $g\in G$ and a random subgroup $\mu$ of $G$, we obtain another random subgroup $g\cdot \mu$ of $G$ whose action on Borel sets is given by $(g\cdot \mu)(B):=\mu(g^{-1}Bg)$, where $g^{-1}Bg=\{g^{-1}Hg \ : \ H\in B\}$.  The random subgroup $\mu$ is said to be \textbf{invariant} if $g\cdot \mu=\mu$ for all $g\in G$.  The acronym \textbf{IRS} is used to stand for an invariant random subgroup. 

For example, a normal subgroup of $G$, viewed as a deterministic random subgroup, is an IRS of $G$.  Another example is given by finite index subgroups:  if $H$ is a finite-index subgroup of $G$, then $H$ has only finitely many conjugates, and the uniform counting measure on these conjugates is an IRS of $G$.

Of particular relevance in the sequel is the example of a \textbf{finitely describable} IRS.  If the group $G$ acts on a finite set $X$, then we get an IRS of $G$ by uniformly randomly picking an element of $x\in X$ and then outputting the stabilizer subgroup $G_x$ of $G$.  The IRS thus obtained is called finitely describable.

An important special case of the Aldous-Lyons conjecture is obtained by asking about the connection between a random subgroup of a finitely generated group and the random rooted graph one obtains by considering the associated Schreier coset graph.  Recall that the Schreier coset graph associated to a subgroup $H$ of a finitely generated group $G$ with respect to a finite generating set $S$ has as vertices right cosets of $H$ and with edges connecting $Hg$ and $Hgs$ for $s\in S$.  Viewing the Schreier coset construction as a function from $\sub(G)$ to the set of rooted graphs, the pushforward of an IRS under this function thus yields a random rooted graph.  Abert, Glasner, and Virag \cite[Proposition 14]{AGV} showed that a random subgroup of a group is an IRS precisely when the associated random Schreier graph is unimodular.  On the other hand, in the case that $G=\bb F_S$, the free group on a finite set $S$, the random Schreier graph (with respect to the generating set $S$) associated to an IRS is sofic if and only if the IRS is a weak* limit of finitely describable IRSs.  Such IRSs are called \textbf{co-sofic}.\footnote{The reason behind the terminology is that for a finitely presented group $G=\bb F_S/N$, one has that $G$ is sofic if and only if the deterministic IRS $N$ is co-sofic \cite[Proposition 6.1]{Gelander}.}  Consequently, the Aldous-Lyons conjecture implies that for every finite set $S$, every IRS of $\bb F_S$ is co-sofic.\footnote{It is worth noting that there are some groups for which are all IRSs are co-sofic and some groups for which there are non-co-sofic IRSs \cite{BLT}.}

\subsection{Refuting the Aldous-Lyons conjecture using $\mre$}
The algebraic consequence of the Aldous-Lyons conjecture described in the previous subsection bears a striking resemblance to Tsirelson's problem.  Indeed, if we let $\irs(G)$ and $\irs_{\sof}(G)$ denote the sets of invariant random subgroups and co-sofic invariant random subgroups of $G$ respectively, then $\irs_{\sof}(G)$ is a subset of $\irs(G)$ and both are closed, convex subsets of the space of probability measures on $\sub(G)$, whereas $C_{qa}(k,n)$ is a subset of $C_{qc}(k,n)$ and both are closed, convex subserts of $[0,1]^{k^2n^2}$.  In the case of Tsirelson's problem, the value of a certain nonlocal game was used to distinguish the closed convex set $C_{qc}(k,n)$ from its closed, convex subset $C_{qa}(k,n)$.  The main idea of the papers \cite{Bowen1} (written by Bowen, Chapman, Lubotzky, and Vidick) and \cite{Bowen2} (written by Bowen, Chapman, and Vidick) is to do something similar in the current context, where the distinguishing linear functional is now defined in terms of a so-called \textbf{subgroup test}.

Throughout, we set $\bb F:=\bb F_S$ for some finite set $S$.  By a \textbf{challenge} we mean a pair $(K,D)$, where $K\subseteq \bb F$ is finite and $D:2^K\to \{0,1\}$ is a function, which determines a continuous function $H\mapsto D(H\cap K):\sub(\bb F)\to \{0,1\}$.  One should view the subgroup $H$ of $\bb F$ as \emph{passing} the challenge if $D(H\cap K)=1$.  A \textbf{subgroup test} is a pair $\cal T:=((K_i,D_i)_{i\in I},\mu)$ consisting of a set of challenges indexed by a finite set $I$ together with a probability distribution on $I$.  The expected value that a given subgroup $H$ of $\bb F$ \emph{passes the test} is given by $\bb E_{i\sim \mu}D_i(H\cap K_i)$.  Any random subgroup $\pi$ of $\bb F$ can be viewed as a \textbf{strategy} for the test $\cal T$, yielding the \textbf{value} $$\val(\cal T,\pi):=\bb E_{H\sim \pi}\bb E_{i\sim \mu} D_i(H\cap K_i)=\int \sum_{i\in I}\mu(i)D_i(H\cap K_i)d\pi(H).$$  Just as one can consider the quantum and quantum commuting values of a nonlocal game, obtained by taking the supremum over quantum and quantum commuting strategies respectively, one can also consider the values of a subgroup test obtained by taking the supremum over all co-sofic IRS strategies and all IRS strategies respectively.  Indeed, define the \textbf{ergodic value}\footnote{The justification behind the terminology is that the supremum must occur at an extreme point of the space of IRSs; such extreme IRSs are usually called \textbf{ergodic} IRSs.} of the test $\cal T$ to be $$\val_{\erg}(\cal T):=\sup\{\val(\cal T,\pi) \ : \ \pi\in \irs(\bb F)\}.$$  One can also restrict to co-sofic strategies, yielding the a priori smaller \textbf{sofic value} $$\val_{\sof}(\cal T):=\sup\{\val(\cal T,\pi) \ : \ \pi\in \irs_{\sof}(\bb F)\}$$ of the test $\cal T$.  Note that, in the definition of the sofic value of a test, one can restrict to the weak* dense subset of finitely described strategies.  Of course, if the algebraic special case of the Aldous-Lyons conjecture is true, then $\val_{\sof}(\cal T)=\val_{\erg}(\cal T)$ for every subgroup test $\cal T$.

The following result is identical to its quantum strategy for nonlocal games counterpart:

\begin{thm}[\cite{Bowen1}, Main Theorem 1]
For any subgroup test $\cal T$, $\val_{\sof}(\cal T)$ is left-c.e. uniformly in $\cal T$ while $\val_{\erg}(\cal T)$ is right-c.e. uniformly in $\cal T$.  Consequently, if the Aldous-Lyons conjecture has a positive answer, then this common value is computable, uniformly in $\cal T$.
\end{thm}

As in the setting of quantum strategies, the left-c.e. result in the previous theorem follows from a brute force search while the right-c.e. result uses semidefinite programming.\footnote{In fact, the right-c.e. result here uses the even simpler theory of linear programming.}

The quest now becomes to show that $\val_{\sof}(\cal T)$ is not computable uniformly in $\cal T$.  The rough idea is to reduce the undecidability of the quantum values of nonlocal games to the sofic values of subgroup tests.  More precisely, to a certain kind of nonlocal game $\cal G$ known as a \textbf{tailored} nonlocal game, the authors effectively associate a subgroup test $\cal T(\cal G)$ in such a way so that:
\begin{enumerate}
    \item If $\cal G$ has a perfect \textbf{$Z$-aligned permutation strategy that commutes along edges}, then $\cal T(\cal G)$ has a perfect finitely described strategy.
    \item If $\cal T(\cal G)$ has an \emph{almost} perfect finitely described strategy, then $\cal G$ has an \emph{almost} perfect quantum strategy.
\end{enumerate}

Here, a perfect $Z$-aligned permutation strategy that commutes along edges is a particular kind of quantum strategy whose definition can be found in \cite[Definitions 6.11 and 6.24]{Bowen1}. In \cite{Bowen2}, the authors prove a version of Theorem \ref{mip} which associates to each Turing machine $\cal M$ a \emph{tailored} nonlocal game $\cal G_M$ so that, in case $\cal M$ halts, the perfect strategy for $\cal G_M$ can be taken to be a $Z$-aligned permutation strategy that commutes along edges (while still having quantum value at most $\frac{1}{2}$ in case $M$ does not halt).

Combining these two results, we see that if $\cal M$ halts, then $\val_{\sof}(\cal T(\cal G_M))=1$.  On the other hand, if $\cal M$ does not halt, then $\val_{\sof}(\cal T(\cal G_M))\leq 1-\lambda(M)$, where $\lambda(M)\in (0,1)$ is some constant that depends computably on the description of $M$.  Consequently, if the Aldous-Lyons conjecture were true and thus the sofic value of a subgroup test was computable uniformly in the description of the test, one could decide the halting problem, yielding a contradiction.

\section{Open questions}

In this section, we collect a number of open questions related to the themes of this paper.

\begin{question}
Is first-order arithmetic interpretable in the theory of the hyperfinite II$_1$ factor $\R$?  Even more specifically, is the universal theory of the integers interpretable in the universal theory of $\R$?
\end{question}

A positive answer to the previous question might allow one to show that the universal theory of $\R$ is not computable without having to invoke $\mre$.

\begin{question}
Is $\val^{co}(\frak G)$ computable uniformly in $\frak G$?
\end{question}

In \cite{MIP}, the authors conjecture that $\operatorname{MIP}^{co}=\operatorname{coRE}$, where $\operatorname{MIP}^{co}$ is defined exactly as $\operatorname{MIP}^*$ but using the quantum commuting value of a game rather than the quantum value and $\operatorname{coRE}$ denotes those languages whose complement is in $\operatorname{RE}$.  The inclusion $\operatorname{MIP}^{co}\subseteq \operatorname{coRE}$ follows from the fact that $\val^{co}(\frak G)$ is right r.e. uniform in $\frak G$.

As mentioned in Subsection \ref{nuclearsection}, the \cstar-algebra version of CEP suitable for stably finite algebras is the MF problem, whose negative resolution follows from the negative resolution of CEP.  But what about a version of CEP suitable for \emph{all} \cstar-algebras?  Kirchberg proposed such a statement, which is now referred to as the \textbf{Kirchberg embedding problem} (KEP):

\begin{question}[KEP]
Does every \cstar-algebra embed into an ultrapower of the \textbf{Cuntz algebra} $\cal O_2$?
\end{question}

The Cuntz $\cal O_2$ algebra is the universal \cstar-algebra generated by two proper isometries, where an isometry in a unital \cstar-algebra $A$ is an element $s\in A$ which satisfies $s^*s=1$; an isometry is proper if it is not a unitary, that is, if $ss^*\not=1$.  The Cuntz algebra plays a central role in the classification theory for \cstar-algebras.  For example, by a famous theorem of Kirchberg, every separable nuclear \cstar-algebra embeds into $\cal O_2$ (with expectation).

A thorough model-theoretic study of the KEP was undertaken in \cite{KEP}, where a number of interesting reformulations were given.  Some computability-theoretic consequences of the KEP were given in \cite{FoxGoldHart}.

The following question was also raised in \cite{tsirelson}:

\begin{question}
Does TP imply QWEP?
\end{question}

While the last question is rather vague, given the contents of the current paper, a positive answer seems rather reasonable:

\begin{question}
   Determine if there is any model theoretic content to the resolution of the Aldous-Lyons conjecture described in Section 6.
\end{question}

\end{document}